\documentclass[11pt]{sat}

\newcommand{\editSF}[1]{{{#1}}}

\newif\ifsattoc\sattoctrue
\newread\testfl\immediate\openin\testfl=\jobname.toc
    \ifeof\testfl\sattocfalse\fi\closein\testfl

\title{\editSF{A Survey of Sampling Discretization  of Integral and Uniform Norms}}
\def\shorttitle{\editSF{A Survey of Sampling  Discretization}}

\author{\editSF{Feng~Dai, Egor~D.~Kosov, and  Vladimir~N.~Temlyakov}\footnote{
		The first named author's research was partially supported by NSERC of Canada Discovery Grant
		RGPIN-2020-03909.
	The second named author's research was supported {by} the AEI grants
	RYC2023-043616-I and PID2023-150984NB-I00 funded by MICIU/AEI/10.13039/501100011033/ FEDER, EU, and by the Spanish State Research Agency, through the Severo Ochoa and Mar\'ia de Maeztu Program for Centers and Units of Excellence in R\&D (CEX2020-001084-M). The second named author thanks CERCA Programme (Generalitat de Catalunya) for institutional support.
The third named author's research (Subsections 3.1, 3.2\editSF{, and} Section 7) was supported by the Russian Science Foundation under grant no. 23-71-30001, https://rscf.ru/project/23-71-30001/, and performed at Lomonosov Moscow State University.
		  }}
\def\shortauthor{\editSF{F.~Dai, E.~D.~Kosov, V.~N.~Temlyakov}}

\def\versiondate{\today}

\def\MSCnumbers{41-02, 41A17, 41A10, 47-02} 

\def\keywords{Marcinkiewicz inequality, sampling discretization of integral norms, universal discretization, entropy numbers } 

\usepackage[centertags]{amsmath}

\usepackage{amsfonts}

\usepackage{amssymb}

\usepackage{latexsym}

\usepackage{amsthm}

\usepackage{newlfont}

\usepackage{graphicx}

\usepackage{listings}

\usepackage{booktabs}

\usepackage{abstract}

\usepackage{enumerate}

\usepackage{xcolor}

\RequirePackage{srcltx}
\usepackage{comment}

\newlength{\defbaselineskip}

\newcommand{\setlinespacing}[1]%
           {\setlength{\baselineskip}{#1 \defbaselineskip}}

\newcommand{\actaqed}{\hfill $\actabox$}

{\medskip\noindent \textit{Proof of #1. }}%
{\actaqed \medskip}

\def\cA{{\mathcal A}}

\def\cC{{\mathcal C}}

\def\cF{{\mathcal F}}

\def\cH{{\mathcal H}}

\def\cX{{\mathcal X}}

\def\bt{\mathbf t}

\def\bx{\mathbf x}

\def\by{\mathbf y}

 \def \<{\langle}

\def\>{\rangle}

\def \La{\Lambda}

\def \Og{\Omega}

\def \e{\varepsilon}

\def \va{\varepsilon}

\def \ff{\varphi}

\def\al{\alpha}

\def\bt{\beta}

\def\ga{\gamma}

\def\la{\lambda}

\def\f{\frac}

\def\bt{\beta}

\newtheorem{Theorem}{Theorem}[section]

\newtheorem{Lemma}[Theorem]{Lemma}

\newtheorem{Proposition}[Theorem]{Proposition}

\newtheorem{Corollary}[Theorem]{Corollary}

\numberwithin{equation}{section}

\theoremstyle{definition}

\newtheorem{Definition}[Theorem]{Definition}

\newtheorem{Remark}[Theorem]{Remark}

\newtheorem{Example}[Theorem]{Example}

\newcommand{\be}{\begin{equation}}

\newcommand{\ee}{\end{equation}}

\usepackage{hyperref}

\usepackage{mathabx}

\def\ga{\gamma}
\def\Ga{\Gamma}
\def\sa{\sigma}

\def\Ld{\Lambda}

\def\Og{\Omega}
\newcommand{\ld}{\lambda}
\def\va{\epsilon}
\def\vi{\varphi}

\def\ZZ{\mathbb{Z}}
\def\EE{\mathbb{E}}

\def\RR{\mathbb{R}}
\def\SS{\mathbb{S}}

\def\NN{\mathbb{N}}
\def\PP{\mathbb{P}}
\def\CW{\mathcal{W}}

\def\Br{\Bigr}
\def\Bl{\Bigl}

\def \spn{\operatorname{span}}

\newcommand{\eq}[1]{\begin{align*}#1\end{align*}}

\def\opn{\operatorname}

\def\mathscr{\mathcal}

\newcommand{\beq}{\begin{equation}}
\newcommand{\enq}{\end{equation}}

\def\D{{\mathcal D}}

\def\CC{\mathbb{C}}

\def\CD{\mathcal{D}}

\DeclareFontEncoding{FMS}{}{}
\DeclareFontSubstitution{FMS}{futm}{m}{n}
\DeclareFontEncoding{FMX}{}{}
\DeclareFontSubstitution{FMX}{futm}{m}{n}
\DeclareSymbolFont{fouriersymbols}{FMS}{futm}{m}{n}
\DeclareSymbolFont{fourierlargesymbols}{FMX}{futm}{m}{n}
\DeclareMathDelimiter{\VT}{\mathord}{fouriersymbols}{152}{fourierlargesymbols}{147}

\def\startpagenumber{1}
\def\volumenumber{10} 
\def\year{2026}

\def\ee{{\rm e}}  

\newcommand{\beginddoc}{
\maketitle
\begin{abstract}
 This paper surveys recent developments in the sampling discretization of integral and uniform norms for functions in general finite-dimensional spaces. These results generalize the classical Marcinkiewicz--Zygmund inequalities for trigonometric and algebraic polynomials, which  play a crucial role in Fourier analysis, interpolation, and approximation theory. We focus on the problem in  the broad context of finite-dimensional subspaces, where norms defined by general probability measures are approximated  by their discrete counterparts.
   The primary emphasis is  on  results  closely related to the authors' recent research.   A key objective  is to highlight   the main ideas and techniques  that form the foundation of the proofs  in this area.  This survey serves as a complement to three recently published  survey papers  on sampling discretization \cite{DPTT, KKLT, LMT}.
\vskip1pt MSC: \MSCnumbers
\ifx\keywords\empty\else\vskip1pt Keywords: \keywords\fi
\end{abstract}
\insert\footins{\scriptsize
\medskip
\baselineskip 8pt
\leftline{Surveys in Approximation Theory}
\leftline{Volume \editSF{\volumenumber}, \year.
pp.~\thepage--\editSF{\pageref{endpage}}.}
\leftline{\copyright\ \year\ Surveys in Approximation Theory.}
\leftline{ISSN 1555-578X}
\leftline{All rights of reproduction in any form reserved.}
\smallskip
\par\allowbreak}
\ifsattoc\else\tableofcontents\fi}
\renewcommand\rightmark{\ifodd\thepage{\it \hfill\shorttitle\hfill}\else {\it \hfill\shortauthor\hfill}\fi}
\def\endddoc{\label{endpage}\end{document}}
\date{{\small \versiondate}}
\RequirePackage{hyperref}
\begin{document}
\beginddoc
\ifsattoc
\bigskip
\def\toczer{0}\def\tochalf{.5}\def\tocone{1}
\def\tocindent{0}
\def\ection{section}\def\ubsection{subsection}
\def\numberline#1{\hskip\tocindent truecm{} #1\hskip1em}
\newread\testfl
\def\inputifthere#1{\immediate\openin\testfl=#1
    \ifeof\testfl\message{(#1 does not yet exist)}
    \else\input#1\fi\closein\testfl}
\countdef\counter=255
\def\diamondleaders{\global\advance\counter by 1
  \ifodd\counter \kern-10pt \fi
  \leaders\hbox to 15pt{\ifodd\counter \kern13pt \else\kern3pt \fi
  \hss.\hss}\hfill}
\newdimen\lextent
\newtoks\writestuff
\medskip
\begingroup
\small
\def\contentsline#1#2#3#4{
\def\argu{#1}
\ifx\argu\ection\let\tocindent\toczer\else
\ifx\argu\ubsection\let\tocindent\tochalf\else\let\tocindent\tocone\fi\fi
\setbox1=\hbox{#2}\ifnum\wd1>\lextent\lextent\wd1\fi}
\lextent0pt\inputifthere{\jobname.toc}\advance\lextent by 2em\relax
\def\contentsline#1#2#3#4{
\def\argu{#1}
\ifx\argu\ection\let\tocindent\toczer\else
\ifx\argu\ubsection\let\tocindent\tochalf\else\let\tocindent\tocone\fi\fi
\writestuff={#2}
\centerline{\hbox to \lextent{\rm\the\writestuff%
\ifx\empty#3\else\diamondleaders{}
\hfil\hbox to 2 em\fi{\hss#3}}}}
\inputifthere{\jobname.toc}\endgroup
\immediate\openout\testfl=\jobname.toc 
\immediate\closeout\testfl             
\renewcommand{\contentsname}{}         
\tableofcontents
\fi



\section{Introduction}

 The classical Marcinkiewicz inequality  has the following original form (see \cite[ Vol. II, p. 30]{Zy}):  For any $1<p<\infty$, there exists a constant $C_p\ge 1$ depending only on $p$ such that for every
 $f\in \mathcal{T}_{n}:=\opn{span} \bigl\{ e^{\mathbf{i}kx}\colon k\in \ZZ, \,  -n\leq k\leq n\bigr\} $,
       \begin{equation}\label{1-1}
    	C_p^{-1}\|f\|_{L_p[0, 2\pi]}\leq\|f\|_{L_p( \Ld)} \leq C_p\|f\|_{L_p[0, 2\pi]},
       \end{equation}
   where  $\Ld=\Lambda_{2n+1}:=\bigl\{ \frac {2j\pi} {2n+1}\colon j=0,1,\ldots, 2n\bigr\}$,
    and $\|f\|_{L_p(\Lambda)}$ denotes the  $L_p$-norm of $f$ defined with respect to the uniform probability measure on $\Ld$; that is,
     \begin{align*}
     \|f\|_{L_p(\Lambda)} :&=\begin{cases}
    \Bl( \frac { 1}{|\Ld|} \sum\limits_{\omega\,\in \Lambda} | f(\omega)|^p\Br)^{1/p},&\  \ \text{if $1\leq p<\infty$,}\\
    \max\limits_{\omega\, \in\Lambda}|f(\omega)|, &\  \ \text{if $p=\infty$}.
    \end{cases}
        \end{align*}
        Throughout this paper, $|E|$ denotes the cardinality of a finite set $E$. It is important to note that the constant $C_p$ on the right-hand side of  \eqref{1-1}  can be chosen to be independent of $p$. Indeed, the right-hand  inequality in \eqref{1-1} remains  true
     at the endpoints $p=1$, $\infty$.  However,   the left-hand inequality requires stronger conditions at the endpoints. It holds  for  the full range of $1\leq p\leq \infty$ with $\Lambda_{4n+1}$ in place of
     $\Lambda_{2n+1}$.

Historically, the inequality \eqref{1-1} for the case  $p=\infty$ (discretization of the uniform norm)
was first established by Bernstein in  1931-1932 ({see \cite{B1} and \cite{B2}}).
The result for $1<p<\infty$  was proved  by J. Marcinkiewicz, while A. Zygmund extended the study to the endpoints    $p=1,\infty$
in 1937 (see \cite[p.4]{KKLT} and \cite[ Vol. II, p. 30]{Zy}).   Therefore, inequalities  of the form~\eqref{1-1} are often referred to as
Bernstein-type theorems when $p=\infty$, and  Marcinkiewicz-type theorems  when $1 \le p < \infty$ (see~\cite{DPTT, KKLT, VT158, Te18}).
In broader literature, these  results are known as Marcinkiewicz--Zygmund inequalities
(see, for  instance, \cite{BD, DPTT, KKLT}).
The Marcinkiewcz type inequality  \eqref{1-1}  is a basic tool in classical analysis, playing a crucial role in the study of Fourier series convergence, Lagrange interpolation, sampling discretization, and weighted approximation.
For various extensions of these results, we refer the reader to  \cite{BD, DPTT, KKLT, Ko-Lu, Lu1, Lu2,  Ma-To} and the references therein.

{ This survey  focuses primarily on recent work by the authors and their collaborators concerning Marcinkiewicz-type and Bernstein-type theorems for general finite-dimensional subspaces.  We do not impose  structural assumptions on these spaces, except for  conditions on the embedding into the space of bounded functions. Our  objective is to highlight the  key techniques and methods used to  establish  inequalities of the form~\eqref{1-1} in this general setting. The primary technical tools rely on concentration inequalities rooted in the work of J.~Bourgain,  J.~Lindenstrauss,  V.D.~Milman,  G.~Schechtman and others (see, e.g.,~\cite{BLM, JS, Mil71} and references thererin), as well as  chaining arguments (see~\cite{Ta}). We aim  to  demonstrate how these techniques apply specifically to the problem of sampling discretization.}

{ It should be noted that this survey is largely confined to results  closely related to the authors' recent research  on sampling discretization.
Given the extensive  literature in this area, the references provided here are by no means complete. Our intention is to focus on the results we are most familiar with, rather than covering topics beyond the immediate scope of our work. In particular,  we do not discuss in detail the important and related applications in sampling recovery. For these topics, we refer the reader to the surveys~\cite{DTSurvRec, KKLT} and the significant recent contributions by D.~Krieg, M.~Ullrich,  T.~Ullrich and their collaborators (see~\cite{BKPU, DKU, JUV, KPUU2, KS, KU, KU2026, NSU, MU} and references therein).}

Let us introduce a general framework for the sampling discretization problem we are going to study.
Let $(\Omega, \mu,\cF)$ be  a  probability space, where $\cF$ is a $\sigma$-algebra of subsets of $\Omega$ and $\mu$ is a probability measure on $\cF$. Given $1\leq p< \infty$, let $L_p(\Omega,\mu)$ denote   the $L_p$-space defined with respect to the measure $\mu$ on $\Omega$, equipped with the norm $\|\cdot\|_{L_p(\Omega, \mu)}$ given by
\[
\|f\|_{L_p(\Omega,\mu)}=\Bigl( \int_{\Omega} |f|^p\, d\mu\Bigr)^{1/p}.
\]
Let $X_N$ denote a linear  space of  $\cF$-measurable  functions on $\Omega$ of dimension $N$ (where $N$ is typically large).
We assume that every  function in  $X_N$ is defined   everywhere on $\Omega$, and satisfies
\[
\|f\|_{L_\infty(\Omega)}:=\sup_{x\in\Omega} |f(x)|{<\infty}.
\]
Finally, we denote the unit ball of $X_N$ in $L_p(\Omega, \mu)$ by
\[
X_N^p:=\bigl\{ f\in X_N\colon \|f\|_{L_p(\Omega, \mu)}\leq 1\bigr\}.
\]

\vskip .1in

{\bf The Marcinkiewicz  discretization  problem with weights.}
Let $C_2\ge C_1>0$ be two given positive  constants.
The weighted  Marcinkiewicz  discretization of the $L_p$ norm, $1\leq p<\infty$, in the space $X_N$ approximates the continuous $L_p(\Omega, \mu)$ norm using
a discrete signed measure supported  on a  finite subset $\Omega_m =\{\xi^1,\ldots, \xi^m\} \subset \Omega$. Specifically, we require that the integral norm of every $f \in X_N$ is bounded by a weighted discrete sum as follows:
\begin{align}
C_1\int_{\Omega} |f|^p\, d\mu\leq  \sum_{j=1}^m \lambda_j |f(\xi^j)|^p\leq C_2 \int_{\Omega} |f|^p\, d\mu,\label{1-2}
\end{align}
where  the nodes  $\xi^j\in  \Og$  and the weights $\lambda_j\in  \RR$ are independent of $f\in X_N$.
{
If the estimate~\eqref{1-2} holds for every function $f$ in  a subspace $X_N$, we say that $X_N$ admits a
{\it weighted Marcinkiewicz-type discretization theorem} with parameters $p$ and
$m \in \mathbb{N}$, and constants $C_1, C_2$.
}

Of particular interest are the following special cases of \eqref{1-2} { where the weights are uniform, i.e.,  $\ld_1=\cdots=\ld_m=\dfrac 1m$.}
	
\begin{enumerate}[\rm (i)]
\item  {\bf The Marcinkiewicz  discretization  problem.}
{
We say that $X_N$ admits a {\it Marcinkiewicz-type discretization theorem}
with parameters $p$ and
$m \in \mathbb{N}$, and constants $C_1, C_2$
if }
\begin{equation}\label{1-3}
C_1  \|f\|_{L_p(\Omega, \mu)}^p \le \frac 1m \sum_{j=1}^m  |f(\xi^j)|^p \le  C_2 \|f\|_{L_p(\Omega, \mu)}^p,\   \  \forall f\in X_N.
\end{equation}
\item {\bf  The Marcinkiewicz  discretization  problem with  $\epsilon\in (0, 1)$.}
{
We say that $X_N$ admits a {\it Marcinkiewicz-type discretization theorem with $\epsilon$} and
parameters $p$ and
$m \in \mathbb{N}$
if }
\begin{equation}\label{1-4}
(1-\epsilon) \|f\|_{L_p(\Omega, \mu)}^p \le \frac{1}{m}\sum_{j=1}^m  |f(\xi^j)|^p \le (1+\epsilon)\|f\|_{L_p(\Omega, \mu)}^p,\   \  \forall f\in X_N.
\end{equation}
	\end{enumerate}


\vskip .1in

{\bf The Bernstein discretization problem.}  The above  definition can also be modified to include the case $p=\infty$, where we ask for
$$
C_1\|f\|_{L_\infty(\Omega)} \leq \max _{1 \leq j \leq m}\left|f\left(\xi^j\right)\right|
\leq\|f\|_{L_\infty(\Omega)},\   \ \forall f\in X_N.
$$

{\bf  Universal Marcinkiewicz-type   discretization problem. }
We can also consider a universal Marcinkiewicz-type   discretization problem for a finite  collection  $\cX:= \{X_k\}_{k=1}^L$ of finite-dimensional   spaces of functions  on $\Omega$. Given $1\leq p<\infty$, we say that a finite set $ \{\xi^j\}_{j=1}^m \subset \Omega $ provides {\it universal discretization} of  the $L_p$ norm for  the collection $\cX$ with  positive constants $C_1, C_2$ if
\[
C_1\|f\|_{L_p(\Omega, \mu)}^p \le \frac{1}{m} \sum_{j=1}^m |f(\xi^j)|^p \le C_2\|f\|_{L_p(\Omega, \mu)}^p,  \   \ \  \forall f\in \bigcup_{k=1}^L X_k.
\]
Universal weighted discretization can be defined likewise.

   In the  framework  described above, $X_N$ can be either a real or complex space.  For Marcinkiewicz discretization  with general constants $C_1, C_2>0$, results for complex spaces can be derived from {the} real case by separating  the real and imaginary parts.  However, this  reduction fails for the Marcinkiewicz discretization problem with $\va>0$, { as the separation argument distorts the isometric constants $(1 \pm \epsilon)$}.  For  simplicity, we  generally   assume  \editSF{that} $X_N$ is a real space unless stated otherwise. Nevertheless, most  results extend  to complex spaces using similar techniques.

\vskip .1in

{This survey presents tools and techniques for the Marcinkiewicz discretization problem on finite-dimensional spaces subject to certain additional conditions.
Typically,  two main types of conditions are used: (i) bounds on entropy numbers;  and (ii)  Nikol'skii-type inequalities.
{
We postpone the definition and discussion of entropy numbers to Section~\ref{entropy}.   In this section,   we focus on {Nikol'skii}-type inequalities. These inequalities constitute the main assumption imposed on the subspaces in our setting and  appear  throughout the paper  in the formulation of various  results.
}

\vspace{3mm}

{\bf Nikol'skii-type inequality.}
Given  $p\in [1,\infty)$, an $N$-dimensional subspace  $X_N\subset L_\infty(\Omega)$ is said to satisfy  the  $(p,\infty)$-Nikol'skii inequality with constant $M>0$ if
\begin{equation}\label{1-6}
\|f\|_{L_\infty(\Omega)} \leq M\|f\|_{L_p(\Omega, \mu)},\   \ \forall f\in X_N,
\end{equation}
in which case we write $X_N\in \textnormal{NI}_{p, \infty}(M)$.  Of particular interest is the case of $p=2$, where the exact constant $M$ in \eqref{1-6} can be written explicitly as follows. Let ${\{u_j(x)\}_{j=1}^N}$ be an orthonormal  basis of $X_N\subset L_2(\Omega, \mu)$. Then we have
\begin{equation}\label{1-6-2026}
\sup_{f\in  X_N, \|f\|_{L_2(\Omega, \mu)}\le 1} |f(x)| = \Bigl(\sum_{j=1}^N {|u_j(x)|^2} \Bigr)^{1/2} =: \sqrt{{F(x)}},\  \ \forall x\in\Omega,
\end{equation}
which implies
\begin{equation}\label{Nik-equiv}
X_N\in \textnormal{NI}_{2,\infty} (M)\iff  \|{F}\|_{L_\infty(\Omega)}\leq M^2.
\end{equation}
The function
\[
\lambda(X_N, \mu; x):=\frac{1}{{F(x)}},
\]
known as the Christoffel function of the space $X_N$, is independent of the selection of orthonormal basis of $X_N$.
{
Finally, since
\[
\|f\|_{L_2(\Omega, \mu)}^2 \le \|f\|_{L_p(\Omega, \mu)}^p \, \|f\|_{L_\infty(\Omega)}^{2-p}
\]
 for $1 \le p \le 2$,  the condition
\[
\|f\|_{L_\infty(\Omega)} \le \sqrt{M}\,\|f\|_{{L_2(\Omega, \mu)}}
\]
implies
\begin{equation}\label{Nik-impl}
	\|f\|_{L_2(\Omega, \mu)} \le M^{\frac{1}{p}-\frac{1}{2}} \|f\|_{L_p(\Omega, \mu)}
	\quad \text{and} \quad
	\|f\|_{L_\infty(\Omega)} \le M^{\f1p} \|f\|_{L_p(\Omega, \mu)}.
\end{equation}
In terms of the Nikol'skii inequalities,  this establishes the implication
\begin{equation}\label{1-7} X_N\in \textnormal{NI}_{2,\infty}({M}^{1/2})\implies
X_N\in \textnormal{NI}_{p,\infty} (M^{1/p})\   \ \text{for any $1\leq p\leq 2$}.
\end{equation}
}


%


{\bf Notation.} To conclude this section, we introduce some notation that will be used throughout this paper.
Given a finite set  of points $\Omega_m=\{\xi^1,\ldots, \xi^m\}$
\editSF{and} $1\leq p\leq \infty$, we denote by $L_p(\Omega_m)$ the space of all functions on $\Omega_m$ equipped with the norm
$$ \|f\|_{L_p(\Omega_m)}:=\begin{cases}
	\Bigl(\dfrac 1m \sum\limits_{j=1}^m|f(\xi^j)|^p\Bigr)^{1/p},&\  \ \text{if $1\leq p<\infty$},\\[5mm]
	\max\limits_{1\leq j\leq m} |f(\xi^j)|, &\   \ \text{if $p=\infty$}.
\end{cases}$$
We  identify  vectors in $\mathbb{R}^m$  with  functions defined  on the set
$$
\Omega_m := \{1, \ldots, m\}.
$$
In this special case, we adopt the following conventions for vectors $f \in \mathbb{R}^m$.  We denote the normalized norm $\|f\|_{L_p(\Og_m)}$ by $\|f\|_{L_p^m}$ and the unnormalized norm by
\[
\|f\|_{\ell_p^m} := m^{1/p} \|f\|_{L_p^m}
=
\begin{cases}
	\Bigl( \sum\limits_{j=1}^m|f(j)|^p\Bigr)^{1/p},&\  \ \text{if $1\leq p<\infty$},\\[5mm]
	\max\limits_{1\leq j\leq m} |f(j)|, &\   \ \text{if $p=\infty$}.
\end{cases}
\]
Accordingly, we define the unit ball
\[
B_{\ell_p^m} := \bigl\{ f \in \mathbb{R}^m\colon \|f\|_{\ell_p^m} \le 1 \bigr\}.
\]
More generally, for a normed space $(X,\|\cdot\|_X)$, we define
\[
B_X := \bigl\{ f \in X\colon \|f\|_X \le 1 \bigr\}.
\]
We also denote the Euclidean norm of $x \in \mathbb{R}^m$  as  $|x|$.
Unless otherwise stated, the symbols $C, c, C_1, c_1, \ldots$ denote positive universal constants.

\vskip .1in

{\bf Organization of the paper.} The remainder of the paper is organized as follows.

In Section 2, we discuss the close connections between the Marcinkiewicz discretization problem and several related areas, including embeddings of finite-dimensional subspaces, frames, spectral properties, and operator norms of submatrices of a given matrix.

In Section 3, we survey recent results on Marcinkiewicz sampling discretization for $L_2$-norms, which is the most extensively studied case.


In Section 4, we present one of the general technical approaches to the Marcinkiewicz sampling discretization of $L_p$-norms for all $1\leq p<\infty$.
It is stated in the form of the  conditional sampling discretization theorem, which highlights the connection between the Marcinkiewicz discretization problem and entropy numbers of compact subsets of bounded functions under the uniform metric. We also outline the key ideas used in the proofs.

In Section 5, we discuss universal discretization,  and its connection with the Restricted Isometry Property, a fundamental concept in compressed sensing.

Section 6 covers recent improvements in bounds for sampling discretization of integral norms, along with the main ideas behind their proofs.

Finally, in Section 7, we briefly discuss some recent progress on sampling discretization in the uniform norm.


\section{Connections with other areas}

The Marcinkiewicz  discretization  problem is closely connected to several other areas.  In this section, we provide a brief overview of these interactions and reformulate the problem in different contexts. We refer to the survey paper  \cite{KKLT} for a detailed discussion  of these connections.

First,
there are deep results on embeddings of finite-dimensional subspaces of $L_p(\Omega, \mu)$ into $\ell_p^N$, which are closely related to the sampling discretization of the $L_p$ norms (see, for instance, \cite{BLM,JS}). Techniques developed for these  embedding problems  have proven highly effective for  sampling  discretization (see \cite{JS}). A detailed  discussion of this connection is available  in Section 4.2 of \cite{KKLT}.

 Second, there is  a close connection between  the Marcinkiewicz  discretization of $L_p$-norms and the concept of
$(\Ld,p)$-frames (see Section 5 of \cite{DKT}).  We  recall the following definition from \cite[Definition 5.2]{DKT}.
\begin{Definition}[{\cite{DKT}}]\label{dfD2}
Let $\La:=\{\la_j\}_{j=1}^m\subset \RR$ and $1\le p<\infty$. A sequence $\{\psi_j\}_{j=1}^m$  of functions in $X_N$    is called  a $(\Lambda, p)$-frame
of $X_N$ with positive constants $A$ and $B$ if for any $f\in X_N$ we have
$$
A\|f\|_{L_p(\Omega,\mu)}^p \le \sum_{j=1}^m \la_j |\<f,\psi_j\>|^p \le B\|f\|_{L_p(\Omega,\mu)}^p,
$$
where	$\langle\cdot,\cdot\rangle$ denotes the inner product of $L_2(\Omega,\mu)$. When $\lambda_1=\cdots=\lambda_m=1$,   a $(\Lambda, p)$-frame is  called a $p$-frame (see \cite{AFR}).
\end{Definition}

Let $\CD(  X_N, \cdot,\cdot)$ denote the reproducing kernel (Dirichlet kernel) of the space $X_N\subset L_2(\Omega, \mu)\cap L_p(\Omega, \mu)$; namely,
\[ \CD(  X_N, x,y):=\sum_{j=1}^N u_j(x) {u_j(y)},\  \ x, y\in\Omega\]
where   $\{u_j\}_{j=1}^N$  is an  orthonormal basis (o.n.b.) of $(X_N,\|\cdot\|_{L_2(\Og, \mu)})$. Note that $\CD(  X_N, \cdot,\cdot)$ is independent of the choice of the orthonormal basis.
The Marcinkiewicz inequality \eqref{1-2} is equivalent to
\[
C_1\|f\|_{L_p(\Omega,\mu)}^p \le \sum_{j=1}^m \lambda_j |\<f,\CD(  X_N, \xi^j,\cdot)\>|^p \le C_2\|f\|_{L_p(\Omega,\mu)}^p,\   \ \forall f\in X_N.
\]
{In other words, in terms  of $(\Ld,p)$-frames, $X_N$ admits a weighted Marcinkiewicz-type discretization theorem of the form \eqref{1-2}
with constants $C_1, C_2>0$
if and only if}
there exist points $\xi^1,\ldots, \xi^m\in\Og$ and weights $\Lambda=\{\lambda_j\}_{j=1}^m\subset \RR$ such that
the sequence $\{\CD(X_N, \xi^j,\cdot)\}_{j=1}^m$ forms a $(\Lambda, p)$-frame of $X_N$ with constants $C_1, C_2>0$.

Third, there is a  probabilistic  approach to  the Marcinkiewicz discretization problem with \editSF{$\va\in~(0, 1)$}. It aims to establish  that for IID random points $\xi^1,\ldots, \xi^m$ drawn from $\mu$ on $\Og$, and $1\leq p<\infty$, the inequality
\beq \label{1-8a} \sup_{f\in X_N^p}  \Bl|\frac 1m\sum_{j=1}^m |f(\xi^j)|^p-\int_{\Og} |f|^p \, d\mu \Br|<\va\enq
holds with high probability. {In particular,  this implies}  that
$X_N$ admits a Marcinkiewicz-type discretization theorem of the form \eqref{1-4}.
This approach connects  to moment estimates of random vectors in probability. Let $\xi$  be  a random point drawn from $\mu$.  Given a basis $\{u_k\}_{k=1}^N$ of $X_N$, consider the vector-valued function  $\mathbf {u} =(u_1,\ldots, u_N)$ on $\Og$, and the random vector  $\pmb \eta=\mathbf {u}(\xi)$. Let  $\pmb\eta^j=\mathbf {u}(\xi^j)$, $j=1, \ldots, m$, be independent copies of $\pmb\eta$.   Since each function $f\in X_N$ has a unique representation
 $ f = \< \mathbf {y}, \mathbf{u}\>$ with  $\mathbf{y}=(y_1,\ldots, y_N)\in \RR^N$, the inequality \eqref{1-8a} above  is equivalent to
 \begin{equation} \label{1-9a} \sup_{\mathbf y \in K} \Bl|\frac 1m \sum_{j=1}^m |\<\mathbf y,  \pmb\eta^j\>|^p -\EE |\< \by,  \pmb\eta\>|^p \Br|< \va,\end{equation}
   where  $K:=\{\by \in\RR^N:\      \EE|\< \by, {\pmb{\eta}}\>|^p\leq 1    \}$ is a symmetric convex set in $\RR^N$ and
    $\< \mathbf{x},  \mathbf y\>$ denotes the dot product $\sum_{j=1}^N x_j y_j$ of $\mathbf x,\mathbf y\in\RR^N$.
 Thus,  the  probabilistic  approach   to  the Marcinkiewicz discretization problem asks:  How many independent copies of a  random vector $\pmb\eta$ are needed to ensure that \eqref{1-9a} holds with high probability?

Fourth, the Marcinkiewicz discretization in finite-dimensional spaces relates  to the spectral properties and operator norms of submatrices of a given matrix. It can be reduced to a corresponding matrix problem.   As above, we fix a basis $\{u_k\}_{k=1}^N$ of $X_N$, and define the vector-valued function  $\mathbf{u} =(u_1,\ldots, u_N)$ on $\Og$.
  Given  a set $\pmb\xi=\{\xi^1,\xi^2,\ldots, \xi^m\}$ of distinct points in $\Og$, consider the $m\times N$  matrix
\[\pmb\Phi(\pmb\xi):= \begin{bmatrix}
	\mathbf u(\xi^1)\\
		\mathbf u(\xi^2)\\
	\vdots\\
		\mathbf u(\xi^m)			
\end{bmatrix}=\begin{bmatrix}
	u_1(\xi^1) &u_2(\xi^1) &\cdots& u_N(\xi^1) \\
	u_1(\xi^2) &u_2(\xi^2) &\cdots& u_N(\xi^2) \\
	\vdots&\vdots&\vdots& \vdots\\
	u_1(\xi^m) &u_2(\xi^m) &\cdots& u_N(\xi^m)
\end{bmatrix}.\]
Given $1\leq p\leq \infty$, define  the following norm on $\RR^N$:
$$
\| \mathbf y\|_{L_p} =\|\<\mathbf y,  \mathbf u\>\|_{L_p(\Og, \mu)},\   \ \by=(y_1,\ldots, y_N)\in\RR^N.
$$
The Marcinkiewicz inequality \eqref{1-3} can  then  be reformulated as
\[
	C_1    \| \mathbf y\|^{{p}}_{L_p}   \leq \bigl\| \pmb\Phi(\pmb\xi) \mathbf y\bigr\|^{{p}}_{L_p^m}\leq  C_2 \| \mathbf y\|^{{p}}_{L_p},\   \  \forall \mathbf y\in\RR^N.
\]
If  $p=2$ and  $\{u_1,\ldots,u_N\}$ is an  orthonormal basis of $(X_N,\|\cdot\|_{L_2(\Og, \mu)} ) $, this is equivalent to
\[
C_1    \| \mathbf y\|_{\ell_2^m}^{{2}} \leq \bigl\| \pmb\Phi(\pmb \xi) \mathbf y\bigr\|^{{2}}_{L_2^m}\leq  C_2    \| \mathbf y\|_{\ell_2^m}^{{2}},\   \  \forall \mathbf y\in\RR^N,
\]
{which can be reformulated in terms of the spectrum of the sampling matrix $\pmb\Phi(\boldsymbol{\xi})$. Specifically, this translates into the following estimates for the maximum and minimum singular values:}
\[
\sa_{\max} (\pmb\Phi(\pmb \xi)) \leq \sqrt{{C_2}m}\    \ \text{and}\    \ \sa_{\min} (\pmb\Phi(\pmb \xi))\ge  \sqrt{{C_1}m}.
\]

If  $\Og=\Omega_M=\{ x_1,\ldots, x_M\}$ is a finite set, then  the vector-valued function $\mathbf u$ can be represented as a matrix with $M$ rows and $N$ columns:
 \[\pmb\Phi:= \begin{bmatrix}
	\mathbf u(x_1)\\
	\mathbf u(x_2)\\
	\vdots\\
	\mathbf u(x_M)			
\end{bmatrix}=\begin{bmatrix}
	u_1(x_1) &u_2(x_1) &\cdots& u_N(x_1) \\
	u_1(x_2) &u_2(x_2) &\cdots& u_N(x_2) \\
	\vdots&\vdots&\vdots& \vdots\\
	u_1(x_M) &u_2(x_M) &\cdots& u_N(x_M)
\end{bmatrix}.\]
In this case, we seek an $m\times N$ submatrix of the full matrix $\pmb\Phi$, {denoted by
$$\pmb\Phi(\pmb\xi):= \begin{bmatrix}
	\mathbf u(\xi^1)\\
	\mathbf u(\xi^2)\\
	\vdots\\
	\mathbf u(\xi^m)			
\end{bmatrix},\  \ \pmb\xi=(\xi^1,\ldots, \xi^m)\in\Og_M^m,
$$ }
such that  for every  $\mathbf{y} = (y_1, \dots, y_N)^T \in \mathbb{R}^N$,
\[C_1 \|\pmb\Phi \mathbf y\|_{L_p(\Omega_M)}^p\leq \|\pmb\Phi(\pmb \xi) \mathbf y\|_{L_p^m}^p \leq C_2 \|\pmb\Phi \mathbf y\|_{L_p(\Omega_M)}^p.\]	
{Here,  setting $f_{\mathbf{y}} := \sum_{j=1}^N y_j u_j$, we have
\[ 	\|\pmb\Phi \mathbf y\|_{L_p(\Omega_M)}^p=\f 1M \sum_{j=1}^M |f_{\mathbf y}(x_j)|^p\    \ \text{and}\  \   \|\pmb\Phi(\pmb \xi) \mathbf y\|_{L_p^m}^p=\f 1m \sum_{j=1}^m |f_{\mathbf y}(\xi^j)|^p.\]}

Finally, the universal Marcinkiewicz-type discretization problem can alternatively be reformulated in terms of certain properties of  matrices. In particular,  the universal discretization of the  $L_2$ norm is closely linked with the concept of the Restricted Isometry Property (RIP), which plays significant roles in compressed sensing. A detailed discussion of this connection will be presented in Section \ref{sec:7} of this paper.

\section{Marcinkiewicz discretization of  $L_2$ norms}\label{subsection General subspaces 2}

%
%
%
%
%
%
%

In this section, we  survey some known results on  the   Marcinkiewicz sampling discretization of the  $L_2$ norms, which is the most studied case.
Let $X_N$ be an $N$-dimensional subspace  of $L_2(\Omega,\mu)$. For simplicity, we assume all functions in $X_N$ are real-valued.
Our goal is  to discretize  the $L_2$ norm in $X_N$ via  a finite discrete sum:
	\begin{equation}\label{4-1:2024}
	(1-\epsilon)	\int_{\Omega} |f|^2\, d\mu\leq \sum_{j=1}^m \lambda_j |f(\xi^j)|^2\leq (1+\epsilon)\int_{\Omega} |f|^2\, d\mu,\   \ \forall f\in X_N,
	\end{equation}
where $\xi^j\in\Omega$, $\lambda_j\ge 0$, $j=1,2,\ldots, m$, and $\epsilon\in (0, 1)$.

\subsection{Characterization  of (\ref{4-1:2024}) via spectral norms}

The discretization \eqref{4-1:2024} of the  $L_2$ norms can be equivalently formulated using spectral norms of certain symmetric matrices. To this end, we first recall  some useful results   from matrix analysis.

Let $A$ be  an $N\times N$ real symmetric matrix.
By the spectral theorem, one can diagonalize $A$ using a sequence of $N$ real eigenvalues,
$$\lambda_1(A)\ge \lambda_2(A)\ge \cdots\ge \lambda_N(A),$$
 and an orthonormal basis of eigenvectors $u_1(A), \ldots, u_N(A) \in \mathbb{R}^N$ such that  $$A \big( u_i(A)\big)=\lambda_i(A)u_i(A),\     \    \   i=1,2,\ldots, N.$$
 Let $\lambda_{\max}(A):=\lambda_1(A)$ and $\lambda_{\min}(A):=\lambda_N(A)$.
Then we have
\begin{equation} \label{5-3-1-0}\lambda_{\max}(A)=\max_{x\in\SS^{N-1}} \langle Ax, x\rangle \   \   \ \text{and}\   \    \lambda_{\min}(A)=\min_{x\in\SS^{N-1}} \langle Ax, x\rangle,\end{equation}
where $\<\cdot,\cdot\>$ denotes the dot product in $\RR^N$.
In general, the Courant--Fischer min-max theorem asserts that
 for each  $1\leq k\leq N$,
	\begin{equation*} \label{5-3-1}\lambda_k(A)= \max_{\dim V =k}\   \  \min_{x\in V, |x|=1} \langle Ax, x\rangle\end{equation*}
	and
	\begin{equation*} \label{5-3-2}\lambda_k(A)=\min_{\dim V=N-k+1}  \   \  \max_{x\in V,  |x|=1}  \langle Ax, x\rangle,\end{equation*}
	where $V$ ranges over all subspaces of  $\RR^N$ with the indicated dimension. If $A$ is positive \editSF{semi-definite},  $\ld_{\max}(A)$  is    the spectral norm of $A$, which coincides with the operator norm $\|A\|$ of $A$:
$$\|A\|:=\max_{x\in \SS^{N-1}}|Ax|.$$
For a general symmetric matrix, we have
$$
\|A\|= \max_{x\in\SS^{N-1}} |\langle Ax, x\rangle| = \max\{|\ld_{\max}(A)|, |\ld_{\min}(A)|\}.
$$

Now we give an alternative formulation of \eqref{4-1:2024} using spectral norms.
Let $\{u_1,\ldots, u_N\}$ be an orthonormal basis of $X_N\subset L_2(\Og,\mu)$, and consider the vector-valued function  $$\mathbf{u}(\xi): =(u_1(\xi),\ldots, u_N(\xi)),\  \    \xi\in \Omega.$$
Define
\[ \mathbf{G}(\xi): =\mathbf{u}(\xi)\otimes \mathbf{u}(\xi) =\Bl[ u_i(\xi) u_j(\xi)\Br]_{1\leq i, j\leq N} \in\RR^{N\times N},\  \ \xi\in\Omega.\]
Then  for each { $f=\sum_{j=1}^N a_j u_j= \< \mathbf{u},  \mathbf{a}\>\in X_N$} with $\mathbf{a}=(a_1,\ldots, a_N)^T\in\RR^N$, we have
\[ \sum_{j=1}^m \lambda_j |f(\xi^j)|^2 -\|f\|_{L^2(\mu)}^2 =\mathbf{a}^T \Bl( \sum_{j=1}^m \lambda_j \mathbf{G}(\xi^j)-I\Br) \mathbf{a}.\]
Here we treat $\mathbf{a}$ as a column-vector and $\mathbf{u}(x)$ as a row-vector.
This together with \eqref{5-3-1-0} yields the following  alternative formulation of \eqref{4-1:2024}.
\begin{Lemma}\label{Lema-4-1}
Let $\xi^1, \ldots, \xi^m\in \Omega$,  $\lambda_1,\ldots, \lambda_m\in \RR$ and $\epsilon\in (0, 1)$. Then
$X_N$ admits a Marcinkiewicz-type discretization theorem of the form
\eqref{4-1:2024}
if and only if
\begin{equation*}
\Bl\|\editSF{I_N} -\sum_{j=1}^m \lambda_j \mathbf{G}(\xi^j)\Br\|\leq \epsilon,
\end{equation*}
where $I$ denotes the $N\times N$ identity matrix, and $\|\cdot\|$ denotes the spectral norm.
 \end{Lemma}

Lemma \ref{Lema-4-1} establishes  a  connection between  the $L_2$-discretization   \eqref{4-1:2024}  and  the following  nonlinear $m$-term approximation problem:
\[ \sigma_m (I, \D)=\inf_{B\in \Sigma_m(\D)} \|\editSF{I_N}-B\|,\]
where $\D:=\{\mathbf{G}(\xi)\colon \xi\in\Omega\}$ and $\Sigma_m(\D)$ denotes the set of all $N\times N$ matrices of the form
\[ B=\sum_{j=1}^m \lambda_j \mathbf{G}(\xi^j),\   \ \lambda_1,\ldots, \lambda_m\in\RR,\  \ \xi^1,\ldots, \xi^m\in\Og.\]

{
Using this connection, one can apply a relaxed greedy algorithm to construct a set of points
$\{\xi^j\}_{j=1}^m \subset \Omega$, with $m$ of order $N^2$, for which~\eqref{4-1:2024}
holds with uniform weights $\lambda_1=\cdots=\lambda_m=\frac{1}{m}$
(see~\cite[Proposition~5.1]{Te18}).
}

\subsection{Discretization using random samples}
We start with the following result, which was derived in \cite{Te18} using Lemma \ref{T5.3} below.

\begin{Theorem}[{\cite{Te18}}]\label{Thm-4-1} Let $X_N$ be an $N$-dimensional space of bounded functions on $\Omega$ such that $X_N\in \textnormal{NI}_{2, \infty}(\sqrt{KN})$
for some constant $K\ge 1$.
Let  $\xi^1,\ldots, \xi^m$ be a sequence of IID random points drawn from  the probability measure $\mu$ on $\Omega$. If $m\ge 3AKN \epsilon^{-2}\log N$ for some constants $A> 1$ and $\epsilon\in (0, 1)$, then
the inequality
\begin{equation}\label{4-7:2024} (1-\epsilon)\|f\|_{L_2(\Omega, \mu)}^2 \leq \frac 1m \sum_{j=1}^m|f(\xi^j)|^2 \leq (1+\epsilon) \|f\|_{L_2(\Omega, \mu)}^2\end{equation}
holds {simultaneously} for all $f\in X_N$ with probability at least $1-2N^{-A+1}$.

\end{Theorem}

{
\begin{Remark} By setting $A=\frac {m\va^2}{3KN\log N}$, we may reformulate the conclusion of Theorem \ref{Thm-4-1}  as follows. If $m\ge 3KN \epsilon^{-2}\log N{+1}$, then  the estimate \eqref{4-7:2024}
holds {simultaneously} for all $f\in X_N$ with probability
at least
$$
1-2N\exp\Bigl(-\frac{m\varepsilon^2}{3KN}\Bigr).
$$
\end{Remark}
}

{We note that  results of this type can be traced back to Rudelson
\cite{Rud1}, whose work  implies a similar discretization bound for  random  points but yields a slightly weaker probability estimate.}

 The proof of Theorem \ref{Thm-4-1}  relies on the following result of  Tropp \cite{Tr}.
\begin{Lemma}[{\cite{Tr}}]\label{T5.3}
Let $\{T_k\}_{k=1}^m$ be a sequence of independent $N \times N$ positive semi-definite random matrices, and define $T = \sum_{k=1}^m T_k$. Assume that$$\max_{1\leq k\leq m}\|T_k\| \le M < \infty \quad \text{almost surely}$$for some constant $M > 0$, where $\|\cdot\|$ denotes the spectral norm.
Let
	$
	\lambda_{\min} := \lambda_{\min}(\EE T)$ and
	$\lambda_{\max} := \lambda_{\max}(\EE T)$.
	Then
	\begin{align*}
	 \PP\Bl[ \lambda_{\min}(T) \le (1-\epsilon)\lambda_{\min}\Br] &\le N\left(\frac{1}{e^{\epsilon}(1-\epsilon)^{1-\epsilon}}\right)^{\lambda_{\min}/M},\  \   \ \forall \epsilon\in [0, 1),\\
	\PP\Bl[\lambda_{\max}(T)  \ge (1+\epsilon)\lambda_{\max}\Br] &\le N\left(\frac{e^{\epsilon}}{(1+\epsilon)^{1+\epsilon}}\right)^{\lambda_{\max}/M},\   \ \forall \epsilon>0.
	\end{align*}
\end{Lemma}
\begin{Remark} { Let $\vi(t):=(1-t) \log (1-t)$ for $t\in [0, 1)$. Since
\[ \vi'''(t)=\f 1{(1-t)^2}>0,\  \  \forall t\in [0,1),\]
it follows by Taylor's theorem that
\[ \vi(t)\ge \vi(0) +\vi'(0) t +\f12 \vi''(0) t^2 =-t+\f {t^2} 2,\  \ \forall t\in [0, 1),\]
which implies
\[ (1-t)^{1-t}\ge \exp\Bl(-t+\f {t^2}2\Br),\   \ \forall t\in [0, 1].\]
Similarly, {applying  Taylor's theorem to the function $\psi(t):=(1+t)\log(1+t)$ at $t=0$, and using the fact that $\psi''(t)=\f 1 {1+t}$ is a decreasing function of $t\ge 0$,} one has
\[ (1+t)\log(1+t) \ge t+\frac 12 \f{t^2} {1+t},\  \ \forall t\geq 0, \]
{which implies }
\[ (1+t)^{1+t} \ge \exp\Bl( t+\frac 12 \f{t^2} {1+t}\Br),\  \ \forall t\geq 0.\]
Substituting these estimates into Lemma \ref{T5.3} yields
\begin{align}
	 &\PP\Bl[ \lambda_{\min}(T) \le (1-\epsilon)\lambda_{\min}\Br] \le N\exp\Bl( -\f {\va^2} 2 \f {\lambda_{\min}} M\Br),\  \   \ \forall \epsilon\in [0, 1),\label{3-4a}\\
	&\PP\Bl[\lambda_{\max}(T)  \ge (1+\epsilon)\lambda_{\max}\Br] \le N\exp\Bl( -\f{\va^2} {2(1+\va)}  \f {\lambda_{\max}}M  \Br),\   \ \forall \epsilon>0. \label{3-5a}
	\end{align} }

\end{Remark}

For completeness, we provide  a proof of Theorem \ref{Thm-4-1} as follows.

\begin{proof}[Proof of Theorem \ref{Thm-4-1}] Let $\{u_1,\ldots, u_N\}$ be an orthonormal basis of $X_N$, and let
\[
\mathbf{u}(x):=(u_1(x),\ldots, u_N(x)) \text{ for } x\in \Omega.
\]
The condition 
$X_N\in \textnormal{NI}_{2, \infty}(\sqrt{KN})$
implies (see \eqref{Nik-equiv})
\[|\mathbf{u}(x)|^2 =\sum_{j=1}^N |u_j(x)|^2\leq KN,\   \ \forall x\in\Omega.
\]
Now consider  the  $m\times N$  random matrix
\[ \pmb \Phi{u}=\begin{bmatrix}
       \mathbf{u}(\xi^1)\\
              \mathbf{u}(\xi^2) \\ \vdots \\
                       \mathbf{u}(\xi^m)
\end{bmatrix}.\]
Let $T:=\frac 1m \pmb \Phi^T\pmb\Phi$. Then
\[ T=\frac 1 {m} \sum_{j=1}^m \mathbf{u}(\xi^j)^T \mathbf{u}(\xi^j)=\sum_{j=1}^m  T_j\   \  \ \text{with}\  \ T_j=\frac 1 {m} \mathbf{u}(\xi^j)^T \mathbf{u}(\xi^j).\] Clearly, the  $T_j$ are independent,  positive semi-definite random matrices  satisfying
\[ \|T_j\|=\max_{x\in \SS^{N-1}} x^T T_j x =\frac 1m|\mathbf{u}(\xi^j)|^2 \leq \frac{KN}m.\]
Moreover, since $\{u_j\}_{j=1}^N$ is an orthonormal basis of $X_N$, we have
\[\EE T =\frac 1m \sum_{j=1}^m\EE \Bl[ \mathbf{u}(\xi^j)^T \mathbf{u}(\xi^j)\Br]=\EE \Bl[ \mathbf{u}(\xi^1)^T \mathbf{u}(\xi^1)\Br] =I_N,\]
where $I_N$ denotes the $N\times N$ identity matrix. Now applying  Lemma  \ref{T5.3}, \eqref{3-4a} and \eqref{3-5a}  with $M=\frac{KN}m$ and $\ld_{\max}=\ld_{\min}=1$ yields  that for any $\epsilon\in (0, 1/2)$,
\begin{align*}
&\PP\Bl[ 1-\epsilon\leq \lambda_{\min} (T)\leq \lambda_{\max} (T) \le 1+\epsilon\Br]\\
&
\ge 1- \PP\Bl[ \lambda_{\min} (T)\le 1-\epsilon\Br]
-\PP\Bl[\lambda_{\max} (T) \ge 1+\epsilon\Br]\\
&\ge 1- N\Bigg[ \exp\Bl( -\f {\va^2} 2 \f {m} {KN}\Br)+  \exp\Bl( -\f{\va^2} {2(1+\va)}  \f {m}{KN}  \Br)\Bigg]\\
&\ge 1-2N \exp\Bl( -\frac {m\epsilon^2}{3KN}\Br)\ge 1-2N^{-A+1}
\end{align*}
provided that $m\ge 3KAN \epsilon^{-2}\log N$.
Theorem  \ref{Thm-4-1} then follows by Lemma~\ref{Lema-4-1}.
\end{proof}
{\begin{Remark} Clearly, the above proof also yields the following estimate:   for any $\va\ge 1 $ and any $m\ge 1$,
\eq{\PP\Bl[ \lambda_{\max} (T) \le 1+\epsilon\Br]
&\ge 1- N \exp\Bl( -\f{\va^2} {2(1+\va)}  \f {m}{KN}  \Br)\\
&\ge1-N \exp\Bl( -\f \va 4 \f {m}{KN}\Br). }
In particular,  if $1\leq m\leq KN\log N$, then taking
 $\va=\f {4 AKN\log N} m$ for some constant $A>1$, we deduce
\eq{\PP\Bl[ \lambda_{\max} (T) \le 1+  \f {4 A  KN\log N} m\Br]\ge 1- N^{-A+1}.
}
Since for any $\by\in\RR^N$,
\[ \< T\by, \by\>  =\f 1m \sum_{j=1}^m |\< \by, \mathbf{u}(\xi^j)\>|^2 \leq \lambda_{\max}(T)\cdot |y|^2,\]
this implies  the following one-sided $L_2$-Marcinkiewicz inequality.
\end{Remark}
\begin{Corollary}
Let $X_N$ be an $N$-dimensional space of bounded functions on $\Omega$, and assume that $X_N\in \textnormal{NI}_{2, \infty}(\sqrt{KN})$
for some constant $K\ge 1$.
Let  $\xi^1,\ldots, \xi^m$ be a sequence of IID random points drawn from  the probability measure $\mu$ on $\Omega$. If  $1\leq m\leq KN\log N$, then for any given constant $A{>} 1 $,
the one-sided  inequality
\[  \frac 1m \sum_{j=1}^m|f(\xi^j)|^2 \leq \Bl( 1+\f {4 A KN\log N} m\Br) \|f\|_{L_2(\Omega, \mu)}^2\]
holds {simultaneously} for all $f\in X_N$ with probability at least $1-N^{-A+1}$.
\end{Corollary}
}

The Nikol'skii condition $\textnormal{NI}_{2, \infty}(\sqrt{KN})$ plays an important  role in the above proof  of   Theorem~\ref{Thm-4-1}.
However, the constant $K$ in this assumption  can be  very large, in which case  one cannot expect a Marcinkiewicz-type theorem with uniform weights and at most $C N \log N$ points with $C>0$ being a universal constant. For instance,  consider the  space $\mathcal{P}_N$ of all algebraic polynomials of degree at most $N-1$ on $[-1, 1]$. It is well known that $\mathcal{P}_N$  satisfies the $(2, \infty)$-Nikol'skii condition $\textnormal{NI}_{2, \infty}(\sqrt{KN})$ with a least constant $K\ge C N$, and any Marcinkiewicz-type inequality  with uniform weights  for the space $\mathcal{P}_N$ requires at least $C N^2$ points (see \cite[Remark~2.1]{Dai}).

Nonetheless, Theorem  \ref{Thm-4-1}  with $K=1$ combined with  a change of density argument  could be used to yield a weighted discretization result for spaces $X_N$ that may not satisfy {the}  $\textnormal{NI}_{2, \infty}(\sqrt{KN})$ assumption or satisfy it only with a very large constant $K$. To be more precise,  let $\lambda(X_N, \mu; \cdot)$ denote the Christoffel function on $\Omega$  associated with the space $X_N$ and the measure $\mu$, defined as
\begin{equation} \label{4-8:eq}\lambda(X_N, \mu; x):=\inf_{\substack{ f\in X_N\\
f(x)=1}} \int_{\Omega} |f(\xi)|^2\, d\mu(\xi),\  \   x\in\Omega,\end{equation}
where it is agreed that $\lambda(X_N, \mu; x)=\infty$ if $f(x)=0$ for all $f\in X_N$.
\editSF{It is} easily seen {from \eqref{1-6-2026}} that $$\int_{\Omega}\frac 1 {\lambda(X_N, \mu; x)} d\mu(x)=N.$$

{Using this notation, we  now derive  a weighted discretization result for more general  spaces $X_N$ without  imposing the  $\textnormal{NI}_{2,\infty}(\sqrt{KN})$ assumption.}

\begin{Corollary}\label{cor-3-1}
Let $X_N$ be an $N$-dimensional subspace  of $L_2(\Omega,\mu)$ with the associated  Christoffel function $\lambda_N(\cdot)=\lambda(X_N,\mu;\cdot)$ as defined in \eqref{4-8:eq}.
Let   $\xi^1,\ldots, \xi^m$ be a sequence of IID random points drawn from the distribution  $\frac 1{N\lambda_N(x)}\, \mu(dx)$. If  $m\ge 3AN \epsilon^{-2}\log N$ for some constants $A>1$ and $\epsilon\in (0, 1)$, then
the inequalities
$$
(1-\epsilon)\|f\|_{L_2(\Omega, \mu)}^2 \leq \frac Nm\sum_{j=1}^m \lambda_N(\xi^j)\left|f\left(\xi^j\right)\right|^2 \leq(1+\epsilon)\|f\|_{L_2(\Omega, \mu)}^2
$$
 hold simultaneously \editSF{for all $f \in X_N$} with probability at least $1-2N^{-A+1}$.
 \end{Corollary}

\begin{proof}
 Let $\left\{u_1, \ldots, u_N\right\}$ be an orthonormal basis of $X_N$. { By \eqref{1-6-2026}}, we have
 $${ \frac 1{\lambda_N(x)} =  |F(x)|^2\   \ \text{with}\  \
F(x):= \big( \sum_{k=1}^N\left|u_k(x)\right|^2\big)^{1/2}.}
$$
Consider the probability measure   $d\nu: ={N^{-1}}F^2  d\mu$ on $\Omega$. Define a  mapping $U: L_2(\Omega, \mu)\to L_2( \Omega,\nu)$   by $Uf(x) ={N}\, f(x)/F(x)$ if  $F(x)>0$, and $Uf(x)=0$ otherwise.
	Since  $F(x)=0$ implies  $f(x)=0$ for all $f\in X_N$, it follows that
	$$\|Uf\|_{L_2( \Omega, \nu)} =\|f\|_{L_2(\Omega, \mu)},\   \ \forall f\in X_N.$$
Moreover,  it is easily seen {from \eqref{4-8:eq}} that
	$$\sup_{x\in\Omega} |g(x)| \leq \sqrt{N}\|g\|_{L_2(\Omega, \nu)},\   \ \forall g\in UX_N.$$
Corollary \ref{cor-3-1}  then follows by  applying Theorem \ref{Thm-4-1} to the space $U X_N\subset L_2(\Omega, \nu)$ with $K=1$.
	\end{proof}

\subsection{The minimum  number of  points required for $L_2$-discretization}\label{subsection-3-3}

{
Given  two constants {$C_2\ge C_1>0$}, we denote by  $m (X_N; 2; C_1, C_2)$  the smallest positive integer $m$ for which  there exist $m$ points $\xi^1,\ldots, \xi^m\in\Omega$  such that
\begin{equation*}
C_1\|f\|_{L_2(\Omega, \mu)}^2 \leq \frac{1}{m} \sum_{j=1}^m |f(\xi^j)|^2 \leq C_2\|f\|_{L_2(\Omega, \mu)}^2,\   \ \forall f\in X_N.
\end{equation*}
Clearly,  $m (X_N; 2; C_1, C_2)\ge N$.
Regarding the upper bound,
Theorem~\ref{Thm-4-1} (originally due to Rudelson~\cite{Rud}) establishes that if $X_N\in \textnormal{NI}_{2, \infty}(\sqrt{KN})$, then
$$
m (X_N; 2; 1/2, 3/2)\leq C KN\log N.
$$
The following remarkable theorem asserts  that
the extra  factor $\log N$ in this bound can be removed.
The proof  relies  on   the breakthrough theorem of A.~Marcus, D.~Spielman, and N.~Srivastava \cite{MSS}
combined with   an iterative procedure  initially proposed  in \cite{BLM, Lu89}.
}

\begin{Theorem} [{\cite{Ko22, LimT}}] \label{thm-4-2}Let $X_N$ be an $N$-dimensional space of bounded functions on $\Omega$ such that $X_N\in \textnormal{NI}_{2, \infty}(\sqrt{KN})$ for some constant $K\ge 1$.  Then for any $\epsilon\in (0,1)$,  there exist {$m$ points}  $\xi^1, \ldots, \xi^m \in \Omega$ with $m \leq C \epsilon^{-2} K N$ {for some  absolute constant $C>0$},   such that
\begin{equation}\label{4-10:2024}
(1-\epsilon)\|f\|_{L_2(\Omega, \mu)}^2 \leq \frac{1}{m} \sum_{j=1}^m |f(\xi^j)|^2 \leq (1+\epsilon) K\|f\|_{L_2(\Omega, \mu)}^2,\   \ \forall f\in X_N.
\end{equation}
\end{Theorem}

We point out that Theorem \ref{thm-4-2} was established  in \cite{LimT} with general positive constants $C_1$ and $C_2K$ instead  of the constants $1-\epsilon$ and $(1+\epsilon)K$ in
\eqref{4-10:2024}. The current version of Theorem~\ref{thm-4-2} is due to~\cite{Ko22}, where it was observed that
inequality~\eqref{4-10:2024} can be obtained by refining the iteration technique introduced in~\cite{LimT}.


{By combining Theorem \ref{thm-4-2} with the change of density argument from the proof of \editSF{Corollary~\ref{cor-3-1}}, we obtain an immediate bound on the required number of evaluation points.}

{\begin{Corollary} [\cite{Ko22}] \label{cor-3-10}Given any  $N$-dimensional space $X_N$  of bounded functions on $\Omega$, and  any $\epsilon\in (0,1)$,  there exist 
\editSF{points $\xi^1, \ldots, \xi^m \in \Omega$  and  weights $\ld_1, \ldots, \ld_m>0$}
{with \editSF{$m \leq C \epsilon^{-2}  N$} for some  absolute constant $C>0$},
such that
\begin{equation*}
(1-\epsilon)\|f\|_{L_2(\Omega, \mu)}^2 \leq  \sum_{j=1}^m \ld_j |f(\xi^j)|^2 \leq (1+\epsilon) \|f\|_{L_2(\Omega, \mu)}^2,\   \ \forall f\in X_N.
\end{equation*}
\end{Corollary}}

{
\begin{Example}\label{exam-3-11}
Let $Q \subset \mathbb{Z}^d$ be a finite set of frequencies, and consider the space
\[
\mathcal{T}(Q) := \operatorname{span}\bigl\{ e^{\mathbf{i} \<k,\cdot\>} : k \in Q \bigr\}
\]
of all trigonometric polynomials with frequencies from $Q$ on the cube
$\Omega = [0, 2\pi)^d$, equipped with the normalized Lebesgue measure $\mu$.
{Since the 
\editSF{moduli}
of functions in the orthonormal basis $\bigl\{e^{\mathbf{i} \<k,\cdot\>} :
k \in Q\bigr\}$ of $\mathcal{T}(Q)$ are uniformly bounded by $1$, the Cauchy-Schwarz inequality implies  that}
  $\mathcal{T}(Q) \in \textnormal{NI}_{2,\infty}(\sqrt{N})$ with
$N := \dim \mathcal{T}(Q) = |Q|$.
Therefore, by Theorem \ref{thm-4-2}, we can find  a set of  $m \leq C \epsilon^{-2} N$
points $\xi^1, \ldots, \xi^m\in \Omega$
such that
\[
(1-\epsilon)\|f\|_{L_2(\Omega, \mu)}^2 \leq \frac{1}{m} \sum_{j=1}^m |f(\xi^j)|^2 \leq (1+\epsilon) \|f\|_{L_2(\Omega, \mu)}^2,\   \ \forall f\in \mathcal{T}(Q).
\]
\end{Example}
}

{
We conclude this section with the following tight estimate on  weighted Marcinkiewicz discretization inequality of $L_2$ norms, which follows directly from a result of G. Schechtman \cite{Sc}.
\begin{Theorem}[\cite{Sc}] \label{Thm-3-11-sc} If $X_N$ is an $N$-dimensional subspace of $L_2(\Omega,\mu)$, then for any $b \in(1,2]$, there exist a set of $m \leq b N$ points ${\xi^1, \ldots, \xi^m} \in \Omega$ and a set of nonnegative weights $\lambda_1,\ldots, \ld_m$ such that
$$
\|f\|_{{L_2(\Omega, \mu)}} \leq\left(\sum_{j=1}^m \lambda_j\left|f\left({\xi^j}\right)\right|^2\right)^{1 / 2} \leq \frac{C}{b-1}\|f\|_{{L_2(\Omega, \mu)}}, \quad \forall f \in X_N,
$$
where $C>1$ is an absolute constant.
\end{Theorem}
Note that in Theorem \ref{Thm-3-11-sc}, the parameter $b>1$ can be chosen to be arbitrarily close to $1$.}

\section{Marcinkiewicz discretization theorems}\label{entropy}

{This section presents general results and methods concerning the discretization of $L_p$ norms for $p \neq 2$.
}

\subsection{A conditional theorem}

Let
$\CW$ be a compact subset of the space of bounded functions on $\Og$.
Let  $\xi^1,\xi^2,\ldots$ be a sequence of independent random points in $\Omega$.  Consider the following  random sampling discretization problem:  for  $\epsilon\in (0, 1)$,  $1\leq p<\infty$ and $ f\in\CW$,
\begin{equation}\label{6-1:2024}
	(1-\epsilon) \|f\|_{L_p(\Omega,\mu)}^p \leq \frac  1m \sum_{j=1}^m |f(\xi^j)|^p\leq (1+\epsilon) \|f\|_{L_p(\Omega,\mu)}^p.
\end{equation}
Without loss of generality, we may assume that
\[
\CW\subset B_{L_p}:=\bigl\{f\in L_p(\Omega, \mu): \|f\|_{L_p(\Omega,\mu)}\leq 1\bigr\}
\]
since otherwise we may replace $\CW$ with the set $\bigl\{ f/\|f\|_{L_p(\Omega,\mu)}\colon f\in\CW\bigr\}.$
Our goal  is  to estimate the probability that \eqref{6-1:2024} holds for all $f\in\CW$ in terms of the   number $m$ of required random points.
Such an estimate can be established in terms of an integral of the $\epsilon$-entropy of the set $\CW$ in the uniform metric. We begin by recalling standard definitions and known estimates concerning entropy (see \cite{CS90,ET96, Ta} and \cite[Chapter 7]{VTbookMA}).

\vskip .1in

\textbf{Entropy.}
 Consider a normed linear space
$X := (X,\|\cdot\|_X)$, and let
\[
B_X := \bigl\{x \in X\colon \|x\|_X \le 1\bigr\}
\]
denote the closed unit ball in $X$.
For any $\epsilon > 0$, the covering number $N_\epsilon(A,X)$ of a compact
subset $A \subset X$ is defined as the smallest positive integer $n$ for which
there exist elements $x^1,\ldots,x^n \in A$ such that
\[
A \subset \bigcup_{j=1}^n \bigl(x^j + \epsilon B_X\bigr).
\]

The $\epsilon$-entropy $\mathcal{H}_\epsilon(A;X)$ of a compact set $A$ in $X$
is defined by
\[
\mathcal{H}_\epsilon(A,X) := \log_2 N_\epsilon(A,X),
\]
while the entropy numbers $e_k(A,X)$ of the set $A$ in $X$ are defined by
\[
e_k(A,X) := \inf\{\epsilon > 0\colon \mathcal{H}_\epsilon(A,X) \le k\},
\qquad k = 0,1,2,\ldots.
\]

Note that in our definition we require $x^j \in A$, which differs from the standard
definition of $N_\epsilon(A,X)$ and $e_k(A,X)$, where this restriction is not imposed.
However, it is well known (see~\cite{VTbookMA}, p.~208) that these characteristics may
differ by at most a factor of~$2$.
To emphasize the dependence on the seminorm $\|\cdot\|_X$, we often write
$e_k(A,\|\cdot\|_X) = e_k(A,X)$ and
$\mathcal{H}_\epsilon(A,\|\cdot\|_X) = \mathcal{H}_\epsilon(A,X)$.

There is a simple and useful bound for the entropy of the unit ball in the finite-dimensional case.
If $X=\mathbb{R}^m$, then for any $0 < \epsilon \le 1$,
\begin{equation}\label{5-2}
m \log_2 \frac{1}{\epsilon}
\le \mathcal{H}_\epsilon(B_X,\|\cdot\|_X)
\le m \log_2 \Bigl(1 + \frac{2}{\epsilon}\Bigr).
\end{equation}
This also implies that
\[
2^{-1} \cdot 2^{-k/m}
\le e_k(B_X,\|\cdot\|_X)
\le 6 \cdot 2^{-k/m},
\quad k \in \mathbb{N}.
\]
Sharp entropy estimates are generally difficult to establish in a general high-dimensional setting. Here we only  formulate a useful  bound  for entropy numbers, which will be used later.

\begin{Theorem}[{\cite[Lemma 4.10, Corollary 4.2]{Ko21}}]\label{thm-5-2}
Let
$\Omega_m=\{\xi^1, \ldots, \xi^m\}\subset \Omega$ be a finite subset of  points in $\Omega$.
Let  $X_N$ be an $N$-dimensional space of bounded functions on $\Omega$ satisfying
\begin{equation}\label{5-10:2024}
\|f\|_{L_\infty(\Omega)} \leq \sqrt{M} \|f\|_{L_2(\Omega, \mu)},\   \ \forall f\in X_N,
\end{equation}
where $M\ge 2$ is a constant.
Then for any $1\leq p \leq 2$,
\begin{equation}\label{5-11:2024}
e_k (X_N^p, \|\cdot\|_{L_\infty(\Omega_m)} ) \leq C_p 
\editSF{(\log m)^{\frac{1}{2}}}
( \log M) ^{\frac 1p-\frac12} \Bl( \frac M {k} \Br)\editSF{^\frac{1}{p}},
\    \     
\editSF{k = 0,1,2,\ldots}
\end{equation}
where  $X_N^p:=\{f\in X_N\colon \|f\|_{L_p(\Omega,\mu)}\leq 1\}$.

\end{Theorem}
\begin{Remark} For $p \in (1,2)$, Theorem \ref{thm-5-2} was established in \cite[Lemma 4.10]{Ko21} via a slight modification of the proof of Proposition 16.8.6 in \cite{Ta}. The case $p=2$ is an immediate consequence of \cite[Corollary 4.2]{Ko21}. Finally, for $p=1$, the conclusion follows from \cite[(3.8) and (3.10)]{DPSTT2} and the following estimate on entropy numbers ({see, for instance, \cite[p. 323, (7.1.6)]{VTbookMA}}):
\[
e_{2k} (X_N^1, \|\cdot\|_{L_\infty(\Og_m)} ) \leq C e_{k} (X_N^1, \|\cdot\|_{L_2(\Og, \mu)}) e_{k}( X_N^2, \|\cdot\|_{L_\infty(\Og_m)}).
\]

\end{Remark}

\begin{Remark} While the dimension $N$ does not appear explicitly  in  the estimate~\eqref{5-11:2024}, the constant $M$ in the Nikolski inequality \ref{5-10:2024} typically depends on $N$, particularly  when $m \gg N$.
Furthermore, the use of the $L_\infty(\Omega_m)$ norm embeds the original space $X_N$ into a subspace of $\mathbb{R}^m$.
\end{Remark}

{
\begin{Remark}\label{rem-ent}
By definition, it is easily seen that  the estimates of entropy numbers
$$e_k(A, X)\le R k^{-\alpha},\  \ k=1,2,\ldots$$
imply
$$\mathcal{H}_\epsilon(A, X)\le c(R/\epsilon)^{1/\alpha},\  \ \va>0,$$
where $R>0$ is a constant independent of $k\in\NN$,  and $c>0$ is an absolute constant.
As a result,   Theorem \ref{thm-5-2}
implies that under the condition 
\editSF{\eqref{5-10:2024}},
for any $1\leq p\leq 2$, we have
\[
\mathcal{H}_\epsilon(X_N^p, \|\cdot\|_{L_\infty(\Omega_m)})
\le
C_p(\log m)^{\frac{p}{2}} (\log M) ^{1-\frac{p}{2}}M\va^{-p},\  \ \va>0.
\]
\end{Remark}	
}

\vskip .1in

We are now ready to formulate the conditional discretization theorem for the
Marcinkiewicz-type discretization inequality of the form~\eqref{6-1:2024}.
The following theorem
extends    results  in  \cite[Theorem~1.3]{DPSTT1},  \cite[Theorem~5.1]{DT} and  \cite{Te18}.

\begin{Theorem}[{\cite[Theorem 6.1]{DTM2}}]\label{prL1}
{Let $(\Omega, \mu,\cF)$ be  a  probability space.}	
Let $\CW$ be a  {compact} set  of bounded functions on $\Omega$  satisfying the following conditions for some $1\leq p<\infty$:
\begin{enumerate}
  \item [\textnormal {(i)}]
  $\sup\limits_{f\in\CW} \sup\limits_{x\in\Omega}  |f(x)|\leq R^{1/p}$ for some  constant $R\ge 1$;
\item [\textnormal {(ii)}]
$\{\lambda f\colon \lambda>0, f\in\CW,  \  \|f\|_{L_p(\Og, \mu)}\leq 1\}=\CW$.
\end{enumerate}
Let $\xi^1, \ldots, \xi^m$ be a sequence of independent random points  on $\Omega$ satisfying
\begin{equation}\label{6-1}
	 \sum_{k=1}^m \PP[\xi^k\in E]=m \cdot \mu(E),\   \ \forall E\in\cF.
	\end{equation}
	Then there exist positive constants $C_p\ge 1$ and $ c_p\in (0, 1)$
depending only on $p$   such that for any  $\epsilon\in (0, 1/2)$ and  any
	integer
	\begin{equation}\label{6-3:2024}
	m\ge  C_p 	  \epsilon^{-5 }   \Bl(\log \frac R\epsilon\Br) \int_{c_p\epsilon}^R \cH_{\epsilon t^{1/p}} (\CW; L_\infty) dt,
	\end{equation}
	the inequality
	\[
		(1-\epsilon) \|f\|_{L_p(\Omega,\mu)}^p \leq \frac  1m \sum_{j=1}^m |f(\xi^j)|^p\leq (1+\epsilon) \|f\|_{L_p(\Omega,\mu)}^p
	\]
	holds {simultaneously} for all $f\in\CW$  with probability at least
	\[ 1-     \exp\Bl(-\frac {c_pm\epsilon^4} {R (\log \frac R\epsilon)^2}\Br).\]
\end{Theorem}

	\begin{Remark}
	Theorem \ref{prL1}   was proved in \cite[Theorem 6.1]{DTM2}, where  \eqref{6-3:2024} is replaced by the following slightly  different
	condition:
	\begin{equation}\label{6-4}
	m\ge  C_p 	  \epsilon^{-5 }\left(\int_{c_p \epsilon^{1/p}} ^{R^{1/p}} u^{\frac  p2-1}    \Bigl(\int_{  u}^{ R^{1/p} }\frac {\cH_{ \epsilon t}(\CW{;} L_\infty)}t\, dt\Bigr)^{\frac 12} du\right)^2.
\end{equation}
{\editSF{Note} that, performing the change of variables $v = u^p$ and $s = t^p$
 yields
\[
\text{RHS of \eqref{6-4}}
=
C'_p\epsilon^{-5 }\left(\int_{c'_p \epsilon} ^{R} v^{-\frac{1}{2}}
\Bigl(\int_{v}^{R}\frac{\cH_{ \epsilon s^{1/p}}(\CW{;} L_\infty)}{s}\, ds\Bigr)^{\frac 12} dv\right)^2,
\]
where $c_p'\in (0, 1)$.
Using the Cauchy--Bunyakovsky--Schwarz inequality
\editSF{and}
the Fubini theorem, we then obtain
}
\begin{align*}
\text{RHS of \eqref{6-4}}& \leq C_p  	\epsilon^{-5}   \Bigl(\log \frac R\epsilon\Bigr)\int_{c_{p}'\epsilon}^R \cH_{\epsilon t^{1/p}} (\CW; L_\infty) dt.\end{align*}
\end{Remark}
\begin{Remark}

\editSF{Either} of the following two conditions implies the condition  \eqref{6-1} in Theorem \ref{prL1}:
	\begin{enumerate}
		\item [\textnormal{(i)}] $\xi^1,\ldots,\xi^m$ are identically distributed according to  $\mu$;
		\item [\textnormal{(ii)}] there exists a partition $\{\Lambda_1, \ldots, \Lambda_m\}$ of $\Omega$ such that $\mu(\Lambda_j)=\frac1m$ and $\xi^j\in\Lambda_j$ is  distributed  according to    $m\cdot \mu\Bl|_{\Lambda_j}$  for each $1\leq j\leq m$.
	\end{enumerate}
	\end{Remark}

Theorem  \ref{prL1} combined with Theorem \ref{thm-5-2} implies Corollary \ref{cor-6-2a} below.

\begin{Corollary}\label{cor-6-2a}
Let $\{V_j\}_{j=1}^L$ be a sequence of linear subspaces of $\RR^n$  which satisfies the following condition for some constant $R\ge 1$:
\begin{equation}\label{6-5b}
\|f\|_{L_\infty^n} \leq\sqrt{R} \|f\|_{L_2^n},\   \    \   \forall f\in \bigcup_{j=1}^L V_j.
\end{equation}
Let $\xi^1, \ldots, \xi^m $ be a sequence of independent random variables taking values in $\{1,2,\ldots, n\}$ and  satisfying
\[ \frac1m \sum_{j=1}^m \PP[\xi^j\in E]=\frac {|E|}n,\   \ \forall E\subset \{1,2,\ldots,n\}.\]
Let $\epsilon\in (0, \frac12)$. Then the following statements hold.

\begin{enumerate}
\item [\textnormal{(i)}] If  $1\leq p\leq 2$, then there exist  constants $C_p\ge 1$ and $c_p>0$  such that for any integer
\begin{equation*}\label{6-6a}
{m \ge C_p \va^{-5}   R \Bl(\log \f R\va\Br)  \Bl[  \va^{-p} \Bl(\log \f R\va\Br) (\log R)^{1-\frac p2}(\log n)^{\frac p2} +  \log L\Br]},
\end{equation*}
the Marcinkiewicz-type discretization inequality
\begin{equation}\label{eq-Marc-4}
(1-\epsilon) \|f\|_{L_p^n}^p \leq \frac  1m \sum_{j=1}^m |f(\xi^j)|^p\leq (1+\epsilon) \|f\|_{L_p^n}^p
\end{equation}
holds {simultaneously} for all $ f\in  \bigcup_{j=1}^L V_j$  with probability \editSF{at least}
$$
1-\exp \Bl( -\frac {c_p m\epsilon^4} {R (\log \frac R \epsilon)^2}\Br).
$$

\item [\textnormal{(ii)}] If $2<p<\infty$, then there exist  constants  $C_p>1$ and $c_p>0$ such that  for any integer
{\[m\ge C_p \epsilon^{-5}   R^{p/2} \, \Bl(\log \f R \va\Br) \Bl[ \va^{-2} \log n +\log L\Br],\]}
the Marcinkiewicz-type discretization	inequality \eqref{eq-Marc-4}
holds {simultaneously} for all $ f\in  \bigcup_{j=1}^L V_j$  with probability \editSF{at least}
$$ 1-\exp \Bl( -\frac {c_p m\epsilon^4} {R^{p/2} (\log \frac R \epsilon)^2}\Br).$$
\end{enumerate}
	
\end{Corollary}

\begin{proof} 	Let
$\Omega_n:=\{1,2,\ldots, n\}$ and let $\mu_n$ denote the uniform distribution on $\Omega_n$. Then  the sequence of random variables $\xi^j$ satisfies the condition \eqref{6-1} with $\Omega=\Omega_n$ and $\mu=\mu_n$.		
Let
\[
\CW_p:=\Bigl\{f\in \bigcup_{j=1}^L V_j:\  \|f\|_{L_p^n}\leq 1\Bigr\}.
\]

(i) Let $1\leq p\leq 2$. 	{Using Theorem \ref{thm-5-2} with $M=R$, $m=n$\editSF{, and}
$\Og_m=\Og_n$, we obtain
\[ e_k(V_j^p, \|\cdot\|_{L_\infty^n}) \leq C_p (\log n)^{\f12}\  (\log R)^{\f 1p-\f12} \Bl( \f R k\Br)\editSF{^{\frac{1}{p}}},
\  \ k\in\NN,\  \ j=1,2,\ldots, L,\]
where $ V_j^p:=\{ f\in V_j:\  \ \|f\|_{L_p^n}\leq 1 \}.$  By Remark~\ref{rem-ent},  this implies
\[ \cH_t (V_j^p; \|\cdot\|_{L_\infty^n})\leq C_p t^{-p} R (\log R)^{1-\frac p2}(\log n)^{\frac p2},\  \ \forall t>0.\]
Since $\CW_p=\bigcup_{j=1}^L V_j^p$, it follows that for any $t>0$,
\begin{align}\cH_t (\CW_p{;} \|\cdot\|_{L_\infty^n}) &\leq \log_2 L +\max_{1\leq j\leq L} \cH_t (V_j^p; \|\cdot\|_{L_\infty^n})\notag\\
&\leq C_p t^{-p} R (\log R)^{1-\frac p2}(\log n)^{\frac p2}  + 2\log L.\label{4-10-2026}
\end{align}
Thus,
\begin{align*}
\int_{c_p\epsilon}^R \cH_{\epsilon t^{1/p}} (\CW_p; \|\cdot\|_{L_\infty^n}) dt
&\leq C_p \va^{-p}  R (\log R)^{1-\frac p2}(\log n)^{\frac p2}  \int_{c_p\epsilon}^R t^{-1} \, dt+  2 R \log L\\
&\leq C_p\epsilon^{-p} \Bl(\log \frac R\epsilon \Br) R (\log R)^{1-\frac p2}(\log n)^{\frac p2} +  2 R \log L.
\end{align*}}

Finally, by \eqref{Nik-impl}, the assumption \eqref{6-5b} implies that  $\|f\|_{L_\infty^n} \leq R^{1/p}\|f\|_{L_p^n}$ for all $f\in\bigcup_{j=1}^L V_j$.
Thus, Statement (i) follows directly by
 applying Theorem \ref{prL1} to
$\CW=\CW_p$.

(ii)  Now assume that  $p>2$.  Then for all $f \in \bigcup_{j=1}^L V_j$, we have
$$\|f\|_{L_\infty^n} \leq \sqrt{R}\|f\|_{L_2^n} \leq R_1^{1/p}\|f\|_{L_p^n}\   \   \text{with}\  \  R_1=R^{p/2}.
$$
Since $\CW_p\subset \CW_2$, we obtain from {\eqref{4-10-2026}} that for any $t>0$,
\begin{align*}
\cH_t (\CW_p{;} \|\cdot\|_{L_\infty^n})&\leq \cH_t (\CW_2{;} \|\cdot\|_{L_\infty^n})\leq C t^{-2} R (\log n) + 2\log L,
\end{align*}
implying that
\editSF{\begin{align*}
\int_{c_p\epsilon}^{R_1}  \cH_{\epsilon t^{1/p}} (\CW_p; \|\cdot\|_{L_\infty^n}) dt 
& \leq C\va^{-2}  R\log n\  \int_{c_p \va}^{R^{p/2}} t^{-\f 2p}\, dt + 2 R^{p/2} \log L\\
&\leq  C_p R^{p/2} \Bl[ \epsilon^{-2}\log n +\log L\Br].
\end{align*}}
Statement (ii) then follows by  applying Theorem \ref{prL1} to
$\CW=\CW_p$.
\end{proof}

%

\subsection{ Preliminary discretizations}

Theorem~\ref{prL1} is often applied together with the entropy estimates stated in
Theorem~\ref{thm-5-2}, as was done in the proof of Corollary~\ref{cor-6-2a}.
This normally requires a preliminary step of discretization involving  more random points.
The following result is useful.

\begin{Proposition}\label{prop-6-2}   Let $1\leq p<\infty$ be a fixed number. Let $\{V_j\}_{j=1}^L$ be a sequence of linear subspaces of $L_\infty(\Omega)$ such that $\dim V_j\leq s$ for all $1\leq j\leq L$, and
\begin{equation} \label{6-8:2024}
\|f\|_{L_\infty(\Omega)}\leq  R^{1/p} \|f\|_{L_p(\Omega, \mu)},\   \ \forall f\in \bigcup_{j=1}^L V_j,   \end{equation}
for some constant $R\ge 1$.
{Assume that    $0< \epsilon\leq \frac 18$, \editSF{that} $m$ is an integer satisfying
\begin{equation}
m\ge 16 R \epsilon^{-2}\Bl[ \log L +s\log \frac 3\epsilon\Br],\label{6-8-a}
\end{equation}
and \editSF{that} $\{\xi^j\}_{j=1}^m$ is a sequence  of IID random points drawn from the distribution $\mu$ on $\Og$. Then
the Marcinkiewicz-type discretization inequality
}

\begin{align}\label{6-9-b}
(1-5\epsilon) \|f\|_{L_p(\Omega, \mu)}^p\leq  \frac 1m \sum_{j=1}^m| f(\xi^j)|^p \leq (1+5\epsilon)\|f\|_{L_p(\Omega, \mu)}^p
\end{align}
holds {simultaneously} for all $f\in\bigcup_{j=1}^L V_j$  with probability  $\ge 1- 2\exp\Bl(-\frac {m\epsilon^2}{ 16 R}\Br)$.
\end{Proposition}

{We point out that in the above proposition, the upper bound $\f18$ in the assumption $0 < \epsilon \leq \frac{1}{8}$ is chosen purely for  convenience, any value in $(0, \f15)$ works by appropriately modifying the 
\editSF{constant~$16$} in \eqref{6-8-a}.}

For the proof of Proposition \ref{prop-6-2}, we need the following two lemmas.  {The first lemma appears  in \cite[Lemma 2.1]{BLM}, but also follows directly from the classical Bernstein inequality
(see~\cite[Theorem~4.22]{VTbookGreedy}).}

\begin{Lemma}[{\cite[Lemma 2.1]{BLM}}]\label{lem-2-2}
	
	Let $\{\eta_j\}_{j=1}^m$ be an independent sequence of  random variables such that $\EE \eta_j=0$, $\EE |\eta_j|\leq M_1$ and $|\eta_j|\leq M_\infty$ almost surely for all $1\leq j\leq m$ and some constants $M_1, M_\infty>0$.  Then for any {$0<\lambda<M_1$},
	\begin{equation*}
\PP\Bl[ 	| \sum_{j=1}^m \eta_j| \ge m\lambda\Br] \leq 2 e^{-\frac{m\lambda^2}{ 4M_1M_\infty}}.
	\end{equation*}

\end{Lemma}

{ A slight  variant of the second lemma  also appears in \cite[Lemma 2.5]{BLM}.}
\begin{Lemma}[{\cite[Lemma 2.5]{BLM}} ] \label{lem-1-2}
		Let $T: X\to Y$ be a bounded linear operator between two normed linear spaces $(X, \|\cdot\|_X)$ and $(Y,\|\cdot\|_Y)$.  Let $\epsilon\in (0, \frac12)$.  Assume that there exists an $\epsilon$-net $\cA\subset B_X$  of   $ B_X$  such that for some constants  $\al, \beta>0$,
	$$ \al \|x\|_X \leq \|T x\|_Y\leq \beta \|x\|_X,\   \   \  \forall x\in \mathcal{A}.$$
	Then\editSF{, with $\ga=\dfrac {3\beta}{\al}$,} we have
	\begin{equation*}\al(1-\ga \epsilon) \|z\|_X\leq \|Tz\|_Y\leq \beta(1+2\epsilon) \|z\|_X\editSF{,\  \ \ \forall  z\in X.}
    \end{equation*}
\end{Lemma}

For completeness, we provide a proof of Lemma \ref{lem-1-2} below.
\begin{proof}
For  $z\in B_X$, we denote by  $\varphi(z)$ an element in $\mathcal{A}$  such that $\|z-\varphi(z)\|_X\leq \epsilon$.  Then  for any $z\in B_X$, we have
\begin{align*}
\|Tz\|_Y \leq \|T\| \epsilon +\|T \varphi(z)\|_Y \leq \|T\|\epsilon +\beta \|\varphi(z)\|_X \leq \|T\|\epsilon + \beta,
\end{align*}
which, taking supremum over all $z\in B_X$,  implies
\[ \|T\|\leq \frac \beta {1-\epsilon}\leq \beta(1+2\epsilon).\]

To prove the left-hand-side estimate, we note that for any $z \in X$ with $\|z\|_X = 1$, we have
\begin{align*}
\|Tz\|_Y& \ge \|T\varphi(z)\|_Y -\|T\| \epsilon \ge \al \|\varphi(z)\|_X -\|T\|\epsilon\\
& \ge \al(1-\epsilon) -2\beta\epsilon\ge \al-3\beta \epsilon.
\end{align*}
{This proves the lemma.}
\end{proof}

\begin{proof}[Proof of Proposition \ref{prop-6-2}  ]    By inequality \eqref{5-2},	
	for each $1\leq j\leq L$,
 there exists  an $\epsilon$-net    $\cA_j\subset V_j^p$  of $V_j^p:=V_j\cap B_{L_p}$ in the space $L_p$  such that
	$|\cA_j| \leq \bigl(1+\frac 2\epsilon\bigr)^{{s}}.$
	Let $\cA:=\bigcup_{j=1}^L  \cA_j.$
	Then
	\begin{align}
	\log |\cA| &\leq \log\Bl[L \Bl(1+\frac 2\epsilon\Br)^{{s}}\Br]\leq  \log L+ {{s}} \log  \frac {3} {\epsilon} .\label{6-12:eq}
	\end{align}
	{Furthermore,  \eqref{6-8:2024} implies that $$ \f {\|f\|_{\infty}^p}{\|f\|_{p}^p} =\Bigg\| \f {|f|^p}{\|f\|_{p}^p}\Bigg\|_\infty \leq R\   \ \text{ for all $ f\in \cA$,}$$
where $\|f\|_\infty=\sup_{x\in\Og}|f(x)|$ and $\|f\|_p=\|f\|_{L_p(\Og,\mu)}$.
Thus,  using \eqref{6-12:eq}, and   Lemma \ref{lem-2-2} with $M_\infty=R$, $M_1=2$ and
$$\eta_j=\f{| f(\xi^j) |^p}{\|f\|_p^p}-1,\   \ j=1,2,\ldots, m,$$
 we  conclude   that the inequalities}
	\begin{equation}\label{6-13:2024} \Bl| \frac 1 m \sum_{j=1}^m| f(\xi^j) |^p- \int_{\Omega} |f|^p\, d\mu\Br| \leq \epsilon\|f\|_{L_p(\Omega, \mu)}^p,\   \ \forall f\in \cA\end{equation}
	hold  with probability
\begin{align*}
\ge &1-2|\cA|\exp\Bl(- \frac {m \epsilon^2} {8R}\Br)=1-2\exp\Bl( \log |\cA|- \frac {m \epsilon^2} {8R}\Br)\\
\ge &1-2\exp\Bl(\log L +{{s}}\log \frac 3\epsilon-\frac {m \epsilon^2} {8R}\Br).
\end{align*}
However,  the condition \eqref{6-8-a} implies
\[ \log L +{{s}}\log \frac 3\epsilon\leq \frac {m \epsilon^2} {16R}.\]
It follows  that
 \eqref{6-13:2024}
  holds  for all $f\in\cA$
  with probability
	\[ \ge 1- 2\exp\Bl(-\frac {m \epsilon^2} {16R}\Br).\]
	To complete the proof, we just need to note that by Lemma \ref{lem-1-2}, \eqref{6-13:2024} implies \eqref{6-9-b}.
\end{proof}

\subsection{Sampling discretization of $L_p$ norms }

{
In this subsection,  we show how to use the preliminary discretization from the last subsection,
together with Corollary~\ref{cor-6-2a}, to deduce a Marcinkiewicz-type discretization theorem.
}

\begin{Theorem}\label{IT1} Let $X_N$ be an $N$-dimensional space of bounded functions on $\Omega$, and assume that
$X_N\in \textnormal{NI}_{2, \infty}(\sqrt{KN})$	
for some constant $K\ge 2$.
Let $\{\xi^j\}_{j=1}^\infty$ be  a sequence of IID
random points   drawn  from the distribution  $\mu$.
Let $\epsilon\in (0, 1/2)$.

\textnormal{(i)} If $1\leq p\leq 2$, then there  exists a constant $C_p(\epsilon)\ge 1$ such that for any given parameter $A\ge 1$, and any integer
\begin{equation} \label{6-14}
m \ge  C_p(\epsilon)  A \cdot  KN \cdot \log^3(2KN ),
\end{equation}
{the Marcinkiewicz-type discretization inequality
\[
(1-\epsilon) \|f\|_{L_p(\Omega, \mu)}^p \leq \frac  1m \sum_{j=1}^m |f(\xi^j)|^p\leq (1+\epsilon) \|f\|_{L_p(\Omega, \mu)}^p
\]
}holds
{simultaneously} for all $f\in X_N$  with probability $\ge 1- (KN)^{-A}$.
	
\textnormal{(ii)}  If $2<p<\infty$, then there  exists a constant $C_p(\epsilon)\ge 1$ such that for any given parameter $A\ge 1$, and any integer
$$
m \ge  C_p(\epsilon)  A \cdot  (KN)^{p/2} \cdot \log^2(2KN ),
$$
{the Marcinkiewicz-type discretization inequality
\[
(1-\epsilon) \|f\|_{L_p(\Omega, \mu)}^p \leq \frac  1m \sum_{j=1}^m |f(\xi^j)|^p\leq (1+\epsilon) \|f\|_{L_p(\Omega, \mu)}^p
\]}holds
{simultaneously} for all $f\in X_N$  with probability $\ge 1- e^{-A}$.	
\end{Theorem}




\begin{proof}

We start with the proof of (i), when $1\leq p\leq 2$. The assumption $X_N\in \textnormal{NI}_{2, \infty}(\sqrt{KN})$ implies (see \eqref{1-7}) that
\[\|f\|_{L_\infty(\Omega)} \leq (KN)^{1/p}\|f\|_{L_p(\Omega, \mu)},\   \ f\in X_N.\]
Thus, without loss of generality, we may assume that
$
m \le C K N^2 \epsilon^{-2} \log \frac{1}{\epsilon},
$
where $C>0$ is a sufficiently large fixed constant, since otherwise
Theorem~\ref{IT1} follows directly from Proposition~\ref{prop-6-2} with
$L=1$, $s=N$, and $R=KN$.
We will use Corollary \ref{cor-6-2a} (i).
To this end, we need a step of preliminary discretization involving more points.

Let $\ell_N$ denote the smallest integer 
\editSF{larger than or equal to $C KN^2\epsilon^{-2}  \log \frac 1\epsilon$}
Let $m_1:= m\ell_N$, and let $\Omega_{m_1}:=\{x_1, \ldots, x_{m_1}\}\subset \Omega$ be a set of IID random points drawn from the distribution   $\mu$. Note that
{
\[
C_1(C, \epsilon)\log (KN)\le \log m_1\le C_2(C, \varepsilon) \log (KN)
\]
for some constants $C_2(C, \epsilon)\ge C_1(C, \epsilon)>0$}. Furthermore, by Proposition  \ref{prop-6-2}, the following  inequalities
hold simultaneously with probability $\ge 1-4\exp \Bl(-\frac {cm_1\epsilon^2}{K N}\Br)\ge 1-4 e^{-mN}$, {for a sufficiently large constant $C>0$},
\begin{align}
&	\Bl| \|f\|_{L_p(\Omega, \mu)}^p  - \|f\|_{L_p(\Omega_{m_1})}^p \Br|\leq \frac \epsilon {10} \|f\|_{L_p(\Omega, \mu)}^p,\   \ \forall f\in X_N, \label{6-15:2024}\\
&	\Bl| \|f\|_{L_2(\Omega, \mu)}^2  - \|f\|_{L_2(\Omega_{m_1})}^2 \Br|\leq \frac \epsilon {10} \|f\|_{L_2(\Omega, \mu)}^2,\   \ \forall f\in X_N. \label{6-16:2024}
\end{align}	
Let $\{ \Lambda_1,\ldots, \Lambda_m\}$ be a partition of the set $\{1,\ldots, m_1\}$ such that $|\Lambda_j|=\ell_N$ for $1\leq j\leq m$.
Let $n_1,\ldots, n_m$ be  an independent sequence of  random variables that is independent of $\{ x_1,\ldots, x_{m_1}\}$ 	
such   that each $n_j$ is uniformly distributed in  $\Lambda_j$. Define
$\xi^j =x_{n_j}$,\  \  $j=1,2,\ldots, m$. Then  \editSF{$\xi^1, \ldots, \xi^m\in~\Omega$} are IID random points with common distribution   $\mu$.

Next, we  fix a set $\Omega_{m_1}$ of  random points $x_1, \ldots, x_{m_1}$ for which  both \eqref{6-15:2024} and \eqref{6-16:2024}  hold. Instead of dealing with  functions on the entire set $\Omega$, we consider their restrictions on the set  $\Omega_{m_1}$ equipped with the uniform distribution.
From the assumption $X_N\in \textnormal{NI}_{2, \infty}(\sqrt{KN})$
and by \eqref{6-16:2024}, we have
\[ \|f\|_{L_\infty(\Omega_{m_1})}\leq \sqrt{2KN} \|f\|_{L_2(\Omega_{m_1})},\  \ \forall f\in X_N.\]
Thus, applying  Corollary \ref{cor-6-2a} (i)  with $L=1$, $n=m_1$ and $R=2KN$,   we conclude that
the inequalities
\begin{equation}\label{6-17:2024}
\Bl| \frac 1m \sum_{j=1}^m |f({x_{n_j}})|^p -\|f\|_{L_p(\Omega_{m_1})}^p \Br|		
\leq  \frac \epsilon {10} \|f\|_{L_p(\Omega_{m_1})}^p  ,\  \ \forall f\in X_N
\end{equation}
hold   with probability
\[\geq 1-  \exp\Bl( -\frac {c_p(\epsilon) m } { KN  \log^2(KN)}\Br)\]
provided that
\[ m\ge C_p(\epsilon) KN  \log^3 (KN).\]
Now, we recall that $\xi^j=x_{n_j}$.
Clearly,
\eqref{6-17:2024} combined with \eqref{6-15:2024} yields that
\begin{equation}\label{6-18:2024}
(1-\epsilon) \|f\|_{L_p(\Omega,\mu)}^p \leq \frac 1m \sum_{j=1}^m |f(\xi^j)|^p  \leq (1+\epsilon) \|f\|_{L_p(\Omega,\mu)}^p,\  \ \forall f\in X_N.
\end{equation}

To conclude the proof,  we estimate the probability of the event  $E$  that \eqref{6-18:2024} holds.
Let $E_1$ denote the event that  both  \eqref{6-15:2024} and \eqref{6-16:2024}  hold. Note that $\chi_{E_1}$  is a function of $(x_1, x_2, \ldots, x_{m_1})$, and $\PP(E_1) \ge  1-4e^{-mN}$.
 Moreover, from the above proof,
 we have
\begin{align*}
	 \chi_{E_1} \cdot \EE[ \chi_E |x_1, \ldots, x_{m_1}]
&\ge \chi_{E_1}\cdot \left( 1-  \exp\Bl( -\frac {c_p(\epsilon) m} { KN  \log^2 (KN)}\Br)\right),
\end{align*}
implying that
\begin{align*}
&	\EE\Bl[\chi_{E_1} \cdot \EE [ \chi_{E} |x_1,\ldots, x_{n_1}]\Br]\ge \PP(E_1) \cdot   \left( 1-   \exp\Bl( -\frac {c_p(\epsilon) m} { KN  \log^2 (KN)}\Br)\right) \\
	&\ge   1-  5  \exp\Bl( -\frac {c_p(\epsilon) m} { KN  \log^2 (KN)}\Br).
\end{align*}
Thus, if $m$ is an integer satisfying \eqref{6-14} for some $A\in\NN$, then
\begin{align*}
\PP (E)& \ge \EE[ \chi_{E\cap E_1}]=
\EE\Bl[\chi_{E_1} \cdot \EE [ \chi_{E} |x_1,\ldots, x_{n_1}]\Br]\\
&\ge 1-5 \exp\Bl( -C_p(\epsilon) c_p(\epsilon) A \log (KN)\Br) \ge 1-(KN)^{-A}
\end{align*}
{provided that the constant $C_p(\varepsilon)>0$ is  sufficiently large.}
This completes the proof of (i).

The proof of Part (ii)  is almost identical to that of Part (i), { with all instances of $KN^2\varepsilon^{-2}\log\frac{1}{\varepsilon}$ replaced by $K^{p/2}N^{1+p/2}\varepsilon^{-2}\log\frac{1}{\varepsilon}$.} { The only difference is that, after the preliminary discretization step, we apply Part (ii) of Corollary \ref{cor-6-2a} instead of Part (i), together with the elementary inequality
\[
\|f\|_{L_2(\Omega,\mu)}\leq \|f\|_{L_p(\Omega,\mu)}, \qquad p>2.
\]
Since no new ideas are involved, we omit the details. }
\end{proof}

\subsection{Sampling discretization  with weights }

The constant $K$ in the $(2, \infty)$-Nikol'skii condition $\textnormal{NI}_{2, \infty}(\sqrt{KN})$ can grow substantially  as the dimension of $X_N$ increases. For instance, if $X_N$ is the space   of all algebraic polynomials of degree less  than $N$ on the interval $[-1,1]$ and $\mu$ is the probability measure on $[-1,1]$ given by { $$d\mu(x)=\f {\Ga(\ld+1)}{\Ga(\ld+\f12) \sqrt{\pi}}(1-x^2)^{\lambda-\frac12}\, dx$$}for some constant  $\lambda>0$, then $X_N\in \textnormal{NI}_{2, \infty}(\sqrt{KN})$ with the least constant $K\sim N^{2\lambda}$ (see, e.g.,  \cite[(7.13), (7.14)]{Ma-To}).  In such instances,  Theorem \ref{IT1} requires a considerably larger number of  random points for effectively discretizing  $L_p$ norms.

Similar to the case of $L_2$-discretization, a change of density argument allows us
to apply Theorem~\ref{IT1} to establish a weighted discretization theorem that does not
require the $(2,\infty)$--Nikol'skii condition.
The proof is similar to that of Corollary \ref{cor-3-1} for $L_2$-discretization. It relies on  the following  Lewis' change of density lemma:

\begin{Lemma}[{\cite[Lemma 7.1]{BLM}}, {\cite{Lew}}]\label{lem-6-2}
Given any $1\leq p<\infty$ and any $N$-dimensional subspace $X_N$ of  $L_p(\Omega, \mu)$,  there exists a basis $\{u_j\}_{j=1}^N$ of $X_N$ such that the function $ F:=\bigl(\frac 1N \sum_{j=1}^N |u_j| ^2\bigr)^{\frac12}$ has the following properties:   $\|F\|_{L_p(\Omega, \mu)}=1$,  and
for all scalars $\{\lambda_j\}_{j=1}^N\subset \RR$,
	\begin{equation}\label{6-19}
		\int_{\Omega} \Bl|\sum_{j=1}^N \lambda_j u_j(x)\Br|^2
		F(x)^{p-2}\, d\mu(x) =\sum_{j=1}^N \lambda_j^2.
	\end{equation}
\end{Lemma}

 To illustrate the idea, let $F$ be the function as given in Lemma \ref{lem-6-2}.
Consider the probability measure  $\nu(dx) = |F(x)|^p \mu(dx)$ on $\Omega$, and the linear  mapping $U: L_p(\Omega, \mu)\to L_p(\Omega,\nu)$ given by
$$ Uf (x)=\begin{cases}
\frac {f(x)} {F(x)},\  \ &\text{if $F(x)\neq 0$};\\
1,\   \  &\text{otherwise}.
\end{cases}$$
Since each function in $X_N$ is identically zero on the set $\{x\colon  F(x)=0\}$, we have
$$\|Uf\|_{L_p(\nu)} =\|f\|_{L_p(\mu)},\   \ \forall f\in X_N.$$
In addition, \eqref {6-19} implies that the sequence
$\psi_j := U u_j$, $j=1,2,\ldots, N$ forms  an orthonormal basis of the space $U X_N$ with respect to the inner product of  $L_2(\Omega, d\nu)$. Moreover, we have
\[
	 \sum_{j=1}^N |\psi_j(x)|^2 =N,\   \ \forall x\in \Omega,
\]
implying
$$\|Uf\|_{L_\infty(\Omega)} \leq N^{\frac12}\|Uf\|_{L_2(\Omega, \nu)},\   \ \forall f\in X_N.$$
Thus,  applying Theorem \ref{IT1} to the space $U X_N$ and the measure $\nu$, we deduce
{the following result}.

\begin{Corollary}\label{cor-6-2} Let $1\leq p<\infty$  and let  $X_N$ be an $N$-dimensional space of bounded functions on $\Omega$ associated with a function  $F$ given in Lemma \ref{lem-6-2}.
{
Consider a new measure $\nu$ given by
\[
d\nu = |F|^p \, d\mu.
\]}Let $\{\xi^j\}_{j=1}^\infty$ be  a sequence of IID
random points drawn  from the distribution $\nu$.
Let $\epsilon\in (0, 1/2)$.  	For  $j=1,2,\ldots, m$, let  $\lambda_j:=\frac 1m \frac 1 {|F(\xi^j)|^p}$ if $F(\xi^j)\neq 0$, and $\lambda_j=\frac 1m$ if $F(\xi^j)=0$.

\textnormal{(i)} If $1\leq p\leq 2$, then there  exists a constant $C_p(\epsilon)\ge 1$ such that for any given parameter $A\ge 1$, and any integer
\begin{equation*} \label{6-21:2024}
m \ge  C_p(\epsilon)  A \cdot  N \cdot \log^3(2N ),
\end{equation*}
{the Marcinkiewicz-type discretization inequalities
}
\begin{equation}\label{6-22}
(1-\varepsilon)\|f\|_{L_p(\Omega, \mu)}^p \leq \sum_{j=1}^m \lambda_j \left|f\left(\xi^j\right)\right|^p \leq(1+\varepsilon)\|f\|_{L_p(\Omega, \mu)}^p
\end{equation}
hold {simultaneously} for all $f\in X_N$  with probability $\ge 1- N^{-A}$.

\textnormal{(ii)}  If $2<p<\infty$,
then there  exists a constant $C_p(\epsilon)\ge 1$ such that for any given parameter $A\ge 1$, and any integer
$$
m \ge  C_p(\epsilon)  A \cdot  N^{p/2} \cdot \log^2(2N ),
$$
{the Marcinkiewicz-type discretization
}
inequality \eqref{6-22}  holds {simultaneously} for all $f\in X_N$  with probability $\ge 1- e^{-A}$.
\end{Corollary}

\section{Universal discretization}\label{sec:7}

{We begin with a general universal discretization result for a finite collection of linear spaces $\{V_j\}_{j=1}^L$ that satisfy an $\textnormal{NI}_{2,\infty}(\sqrt R)$-condition with a common constant $R\ge 1$.}

\begin{Theorem}\label{Theorem-7-1}
Let $\{V_j\}_{j=1}^L$ be a sequence of linear subspaces of $L_\infty(\Omega)$  which satisfies the following condition for some constant $R\ge 1$:
\begin{equation}\label{5-1:2026}
\|f\|_{L_\infty(\Omega)} \leq\sqrt{R} \|f\|_{L_2(\Omega, \mu)},\   \    \   \forall f\in \bigcup_{j=1}^L V_j.
\end{equation}
Let $\xi^1, \ldots, \xi^m $ be a sequence of  independent random variables  taking values in   $\Omega$ and  satisfying
	\[ \frac1m \sum_{j=1}^m \PP[\xi^j\in E]=\mu(E),\   \ \forall E\in\cF.\]
	Let $\epsilon\in (0, \frac12)$ and  $1\leq p\leq 2$.  Then there exist  constants $C_p(\epsilon)\ge 1$ and $c_p>0$  such that for any parameter $A\ge 1$ and  any integer
	\begin{equation}\label{7-2:2026}
	m \ge C_p(\epsilon)\cdot  A  \cdot  R   \Bl[(\log R)^3+   (\log R)^{3-\frac p2}(\log \log L)^{\frac p2} +   (\log L)(\log R) \Br],
	\end{equation}
	the Marcinkiewicz-type discretization inequality
\[
	(1-\epsilon) \|f\|_{L_p(\Omega, \mu)}^p \leq \frac  1m \sum_{j=1}^m |f(\xi^j)|^p\leq (1+\epsilon) \|f\|_{L_p(\Omega, \mu)}^p
\]		
holds \editSF{simultaneously} for all $ f\in  \bigcup_{j=1}^L V_j$  with probability $\ge 1-R^{-A}$.
	\end{Theorem}

{To the best of our knowledge, Theorem \ref{Theorem-7-1} is new for general linear subspaces of $L_\infty(\Og)$. The proof proceeds similarly to that of Theorem \ref{IT1}. We summarize it below.}

\begin{proof}

{Note first that in the 
\editSF{case} where  $\mu$ is  the uniform probability measure on \editSF{$\Omega~=~\{1, 2, \ldots, n\}$} with $n \leq C_\epsilon R^2 \log L$,   Theorem \ref{Theorem-7-1}
 follows directly from Part (i) of Corollary \ref{cor-6-2a}.}

{For a general probability measure $\mu$, the proof of Theorem \ref{Theorem-7-1} proceeds in two steps. In the first step, we observe  from \eqref{5-1:2026} and \eqref{1-6-2026}  that $\dim V_j \leq R$ for each $1 \leq j \leq L$. Therefore, for $1 \leq p \leq 2$, applying Proposition \ref{prop-6-2} with $s = R$  yields a preliminary  universal discretization  using  $$m \approx C_\epsilon (R \log L + R^2)$$ random points. Although this sample size estimate is weaker than the bound stated in \editSF{Theorem~\ref{Theorem-7-1}},  it provides the crucial initial reduction step.}

{In the second step, we fix the random points obtained from the preliminary discretization
\editSF{and} restrict all functions on this  set.  We then  identify each $V_j$ with a linear subspace of $\mathbb{R}^n$, where  $$n \approx R \log L + R^2.$$  Thus,  Part (i) of Corollary \ref{cor-6-2a} can be applied   to further reduce the  sample size. The technical details closely follow those in the proof of Theorem \ref{IT1}.}
\end{proof}

{We note that the sample size estimate \eqref{7-2:2026} in Theorem \ref{Theorem-7-1} is optimal up to logarithmic factors when the number of linear spaces $L$ is relatively small (e.g., $L \leq R^\al$ for some constant $\al \ge 1$). To illustrate this, consider the collection of  $N$-dimensional subspaces of trigonometric polynomials $\{\mathcal{T}(Q_j)\}_{j=1}^{N^2}$ from  Example \ref{exam-3-11}, where $Q_j\subset \mathbb{Z}^d$ and $|Q_j|=N$.    Theorem \ref{Theorem-7-1} guarantees universal discretization of the $L_p(\mathbb{T}^d)$ norm for $1 \leq p \leq 2$ and all $f \in \bigcup_{j=1}^{N^2} \mathcal{T}(Q_j)$ whenever the sample size satisfies $m \ge C N\log^3 N$.}

{On the other hand, the estimate in Theorem \ref{Theorem-7-1} can be sharpened  if  the collection  $\{V_j\}_{j=1}^L$ has additional structure, such as being  generated by a finite set of uniformly  bounded functions on $\Omega$,   and $L$ is substantially larger than $R$ (e.g.,   $L \ge R^R$).  For the remainder of this section, we focus precisely on this specific setting.}

We introduce the general setting as follows.
Let
$\D_N=\{\varphi_i\}_{i=1}^N$ be a  fixed system  of linearly independent  bounded  functions on $\Omega$ satisfying
\[
\sup_{x\in\Omega} |\varphi_j(x)|\leq 1,\   \ j=1,2,\ldots, N.
\]
Given an integer $1\leq s\leq N$, we denote by $\mathcal{X}_s(\D_N)$ the collection of all linear spaces spanned by  $\{\varphi_j\colon j\in J\}$  with $J\subset \{1,2,\ldots, N\}$ and $|J|=s$. A function  $f:\Omega\to \CC$ is said to be $s$-sparse with respect to  $\CD_N$ if it belongs to a  linear space from the collection $\mathcal{X}_s(\D_N)$.
We
denote by $\Sigma_s(\D_N)$  the set of all $s$-sparse functions with respect to the dictionary $\CD_N$; namely,
\begin{align*}
	\Sigma_s(\D_N):&= \bigcup_{V\in\cX_s(\D_N)} V.
\end{align*}
Recall that a finite sequence $\{\xi^j\}_{j=1}^m$ of points in $\Omega$ is said to provide
universal discretization of the $L_p$ norm for $\Sigma_s(\mathcal{D}_N)$,
$1 \le p < \infty$, with positive constants $C_1, C_2$ if
\begin{equation}\label{7-4}
C_1\|f\|_{L_p(\Omega, \mu)}^p \le \frac{1}{m} \sum_{j=1}^m |f(\xi^j)|^p \le C_2\|f\|_{L_p(\Omega, \mu)}^p,\
\ \forall f\in\Sigma_s(\D_N).
\end{equation}

Consider a sequence $\{\xi^j\}_{j=1}^m$ of IID random points with common distribution $\mu$.
Our goal is to establish an estimate for  the probability that  \eqref{7-4}  holds in terms of   the number $m$ of random points.
In the above context, we have the following result, which  improves  the estimates  in \editSF{Theorem~\ref{Theorem-7-1}}.

\begin{Theorem}[\cite{DTM2}]\label{thm-7-2}  Let $ \CD_N=\{\ff_j\}_{j=1}^N$ be a set of $N$ bounded functions on $\Omega$
satisfying
\[
\max_{1\leq j\leq N} \|\ff_j\|_{L_\infty(\Omega)}\le 1.\]
Let $1\leq s\leq N$ be a given integer. Assume that there exists a constant $K\ge 1$ such that
\begin{equation}\label{7-2}
\sum_{j\in J} |a_j|^2 \le K  \Bigl\|\sum_{j\in J} a_j\ff_j\Bigr\|^2_{L_2(\Omega, \mu)},\   \   \ \forall a_j\in\CC
\end{equation}
whenever  $J\subset \{1,2,\ldots, N\}$ with $|J|=s$.
Let $\xi^1,\ldots, \xi^m$ be IID
random points  with common  distribution  $\mu$ on $\Omega$.
Then given  $1\le p\le 2$ and $\epsilon\in (0,\frac12)$,
there exist constants $C_p(\varepsilon)>1$ and $c_p(\varepsilon)>0$,
{depending only on $p$ and $\varepsilon$}, such that the inequalities
\begin{equation*}\label{7-8:2024}
(1-\epsilon)\|f\|_{L_p(\Omega, \mu)}^p \le \frac{1}{m}\sum_{j=1}^m |f(\xi^j)|^p \le (1+\epsilon)\|f\|_{L_p(\Omega, \mu)}^p,\   \   \ \forall f\in  \Sigma_s(\CD_N)
\end{equation*}
hold with probability $\ge 1-2 \exp\Bl( -\frac {c_p(\epsilon) m}{Ks\log^2 (2Ks)}\Br)$ provided that
$$
m \ge  C_p(\epsilon) Ks \log N\cdot (\log(2Ks ))^2\cdot (\log (2Ks )+\log\log N).
$$
\end{Theorem}

Theorem \ref{thm-7-2} was proved in \cite{DTM2}.  Note that  the problem of universal discretization for some subspaces of the
trigonometric polynomials was studied in \cite{DPTT, VT160}.

The condition \eqref{7-2} is satisfied for all $1\leq s\leq N$ if  $\varphi_1, \ldots, \varphi_N$ are linearly independent in $L_2(\Omega, \mu)$, in which case
one can choose  $K^{-1}$ to be the smallest eigenvalue of the $N\times N$
matrix $\Bl[ \langle \varphi_j, \varphi_k\rangle_{L_2(\Omega, \mu)}\Br]_{1\leq j, k\leq N}$.
In applications, we often  assume that $\CD_N$ is a Riesz basis; that is,  there exist constants $0< R_1 \le R_2 <\infty$ such that for   any  $(a_1,\ldots, a_N)\in\CC^N$,
\begin{equation*}\label{Riesz0}
R_1 \Bigl( \sum_{j=1}^N |a_j|^2\Bigr)^{1/2} \le \Bigl\|\sum_{j=1}^N a_j\ff_j\Bigr\|_{L_2(\Omega, \mu)} \le R_2 \Bigl( \sum_{j=1}^N |a_j|^2\Bigr)^{1/2}.
\end{equation*}
Note however that estimates in Theorem \ref{thm-7-2}  are independent of the constant~$R_2$.

The universal discretization of $L_2$ norm  is closely related to  the concept of Restricted Isometry Property (RIP), which plays a central role in compressed sensing {(see, e.g., \cite{FR} and the references therein)}.
Given an integer $1\leq s\leq N$, we denote by $\Sigma_s(\RR^N)$ the set of all $s$-sparse vectors in $\RR^N$; i.e., vectors $\bx=(x_1,\ldots, x_N)\in\RR^N$ with at most $s$ nonzero  coordinates $x_i$.
An  $m \times N$ matrix $\mathbf{A}$ is said to satisfy the RIP  of order $s$  if  there exists a positive constant $\delta \in(0,1)$ such that
\begin{equation} \label{7-5}(1-\delta)\|\mathbf{x}\|_{\ell_2^N}^2 \leq\|\mathbf{A} \mathbf{x}\|_{\ell_2^m}^2 \leq(1+\delta)\|\mathbf{x}\|_{\ell_2^N}^2,\   \ \forall \mathbf{x} \in \Sigma_s(\RR^N).
\end{equation}
The smallest positive constant $\delta$ that satisfies \eqref{7-5},  denoted by $\delta_s(\mathbf{A})$, is called  the restricted isometry constant.

{Results on the RIP properties of the random matrix $\pmb\Phi(\pmb\xi)$ associated with uniformly bounded orthonormal systems $\mathcal{D}_N$ can be found in \cite{FR}.
Generalizations of the RIP, where one or both of the
$\ell_2$ norms are replaced by other norms, typically $\ell_p^N$ norms, have also proven
useful in this context (see, e.g.,~\cite{BGIKS, ChS08}, \cite[Exercise~9.6]{FR}).
To express universal sampling discretization within the RIP framework,
the following extension  was introduced in~\cite[Section~5]{VT202v2}.	
}

\begin{Definition} Let $\|\cdot\|$ be a norm on $\mathbb{R}^N$, and let $1\leq p<\infty$. Let $1\leq s, m\leq N $ be integers. An $m\times N$ real matrix $\mathbf {A}$   is said to have  the $s$-th order $R I P\left(\ell_p^m,\|\cdot\|\right)$  with constants $C_1, C_2>0$ if for any $\mathbf{x} \in \Sigma_s(\RR^N)$,  we have

$$
C_1\|\mathbf{x}\|^p \leq\|\mathbf{A x}\|^p_{\ell_p^m} \leq  C_2\|\mathbf{x}\|^p.
$$
\end{Definition}

Next, let $\Phi:\Omega\to\RR^N$ be a vector-valued function on $\Omega$  given by
	\[ \Phi(\bx) =(\varphi_1(\bx),\ldots, \varphi_N(\bx)),\   \ \bx\in\Omega.\]
Here and elsewhere in this section, we always treat $\Phi(\bx)$ as a row vector.
	Given  a finite sequence  $ \pmb\xi:=\{\xi^j\}_{j=1}^m\subset \Omega$ of   points, consider the $m\times N$ random matrix
	
	\begin{equation}\label{6-9-matrix} \pmb\Phi(\pmb\xi):= \begin{bmatrix}
		\Phi(\xi^1)\\
		\Phi(\xi^2)\\
		\vdots\\
		\Phi(\xi^m)			
	\end{bmatrix}
=\begin{bmatrix}
		\varphi_1(\xi^1) &\varphi_2(\xi^1) &\cdots& \varphi_N(\xi^1) \\
		\varphi_1(\xi^2) &\varphi_2(\xi^2) &\cdots& \varphi_N(\xi^2) \\
		\vdots&\vdots&\vdots& \vdots\\
		\varphi_1(\xi^m) &\varphi_2(\xi^m) &\cdots& \varphi_N(\xi^m)
	\end{bmatrix}.\end{equation}
	Additionally, given $1\leq p<\infty$,  we define the norm $\|\cdot\|_{p,\Phi}$ on $\RR^N$ by
	\[ \| \mathbf a\|_{p,\Phi} =\Bl\|\sum_{j=1}^N a_j \varphi_j\Br\|_{L_p(\Omega, \mu)},\   \   \mathbf{a} =(a_1,\ldots, a_N)^T\in \RR^N.\]
With the above notation, we can then write  the inequality \eqref{7-4}  in the form
	\[ C_1    \| \mathbf a\|_{p,\Phi}   \leq \bigl\| \pmb\Phi(\pmb \xi) \mathbf a\bigr\|_{L_p^m}\leq  C_2    \| \mathbf a\|_{p,\Phi},\   \  \forall \mathbf a\in \Sigma_s (\RR^N), \]
which is equivalent to asserting  that matrix $m^{-1/p}\pmb\Phi(\pmb\xi)$ has  the $s$-th order  $R I P\left(\ell^m_p,\|\cdot\|_{p,\Phi}\right)$ with  constants $C_1, C_2>0$.
In particular, if $p=2$, $\epsilon\in (0, 1)$ and   $\varphi_1,\ldots,\varphi_N$ is an orthonormal system in $L_2(\Omega, \mu)$, then  the universal discretization \eqref{7-4} with $C_1=1-\epsilon$ and $C_2=1+\epsilon$
can be formulated equivalently in terms of the RIP property of order $s$ of the random matrix $m^{-1/2}\pmb\Phi(\pmb \xi)$:
\begin{equation*}  (1-\epsilon)\|z\|_{\ell_2^N}^2\leq m^{-1/2}\|\pmb\Phi(\pmb\xi) z\|_{\ell_2^m}\leq (1+\epsilon) \|z\|_{\ell_2^N}^2,\    \ \forall z\in \Sigma_s(\RR^N).
\end{equation*}

In summary, 
\editSF{we state the following result.}
\begin{Theorem}\label{thm-6-3a}  Assume that
$\D_N=\{\varphi_1,\ldots,\varphi_N\}$  is  a finite  system  of  real-valued   functions on $\Omega$,  $1\leq p<\infty$ and $1\leq s, m\leq N$ are integers.   Then   a sequence $ \pmb\xi:=\{\xi^j\}_{j=1}^m$  of $m$ points  in $\Og$  provides 
\editSF{the {universal discretization}} \eqref{7-4} of  the $L_p$ norm for $\Sigma_s(\D_N)$   with   constants $C_1, C_2>0$ if and only if the $m\times N$ matrix $m^{-1/p}\pmb\Phi(\pmb\xi)$, with  $\pmb\Phi(\pmb\xi)$ being defined  in \eqref{6-9-matrix},  has  the $s$-th order  $R I P\left(\ell^m_p,\|\cdot\|_{p,\Phi}\right)$ with  constants $C_1, C_2>0$. In particular, if  $p=2$, $\epsilon\in (0, 1)$ and   $\D_N$ is an orthonormal system in $L_2(\Omega, \mu)$, then  the universal discretization \eqref{7-4} with constants $C_1=1-\epsilon$ and $C_2=1+\epsilon$
holds if and only if the matrix  $m^{-1/2}\pmb\Phi(\pmb\xi)$ has
the RIP property of order $s$ with constants $1\pm \epsilon$.
\end{Theorem}

Some applications of  universal discretization and  Theorem \ref{thm-6-3a} can be found in   \cite{DTM3,DTM1}.

    	\section{Some improved bounds in sampling  discretization of integral norms}

{Analogous to the $L_2$-discretization discussed in Subsection \ref{subsection-3-3}, for any given constants $C_2\ge C_1>0$ and $1\leq p<\infty$, we define  $m (X_N; p; C_1, C_2)$  to be  the smallest positive integer $m$ for which  there exist $m$ points $\xi^1,\ldots, \xi^m\in\Omega$ such that
\[
C_1\|f\|_{L_p(\Omega, \mu)}^p \leq \frac{1}{m} \sum_{j=1}^m |f(\xi^j)|^p \leq C_2\|f\|_{L_p(\Omega, \mu)}^p,\   \ \forall f\in X_N.
\]
For simplicity, we denote this quantity by $m(X_N; p)$ when the specific values of the absolute constants $C_1, C_2$ are  understood or not important from context.
}
    	
{According to Theorem \ref{IT1}, for every space $X_N$ satisfying the $(2, \infty)$-Nikol'skii inequality $\textnormal{NI}_{2, \infty}(\sqrt{KN})$},
we have
\begin{equation}\label{8-3}
m (X_N; p)\leq C_p\begin{cases}
KN \log^3(KN),&\  \ \text{if $1\leq p\leq 2$};\\
(KN)^{p/2} \log^2 (KN), &\  \ \text{if $2<p<\infty$}.
\end{cases}
\end{equation}
The  upper bounds in \eqref{8-3} turn out to be  nearly optimal as $N\to \infty$	for a general space $X_N\in \textnormal{NI}_{2, \infty}(\sqrt{KN})$.
Indeed, this is clear for $1\leq p\leq 2$ as the lower bound  $m (X_N; p)\ge N$
holds trivially.   For $2<p<\infty$,  it is known (see, for instance,  \cite[{\bf D.20}]{KKLT})} that for  $X_N=\spn\{ r_1, \ldots, r_N\}$ with     	
$r_j(t):=\text{sgn} (\sin ( 2^{j+1}\pi t))$, $t\in [0, 1]$, $j\in\NN$  being the sequence of  Rademacher functions on $[0, 1]$ and for the ususal Lebesgue measure on $\Omega=[0, 1]$,  the $(2, \infty)$-Nikol'skii inequality $\textnormal{NI}_{2, \infty}(\sqrt{KN})$ is satisfied   with $K=1$, and  $m (X_N; p)\ge c N^{p/2}$  for any $2<p<\infty$,  where $c>0$ is an absolute constant.

    {	
It is natural to ask  whether the upper bounds in  \eqref{8-3} remain valid  without the extra  logarithmic  factors under the condition $X_N\in \textnormal{NI}_{2, \infty}(\sqrt{KN})$.  For the  discretization of $L_2$ norm, the answer  is positive, due to the breakthrough work of A.~Marcus, D.~Spielman, and N.~Srivastava \cite{MSS}. In fact, according to
 Theorem \ref{thm-4-2}, we have
\[	
m (X_N; 2; 1/2, 3/2)\leq C KN.
\]
For $1\leq p<2$, this remains a   challenging open problem.}

Recent progress on this problem for $1 \leq p < 2$ was established in \cite{DKT}. 
\editSF{Specifically, for }
$1\leq p<2$ and every $X_N\in \textnormal{NI}_{2, \infty}(\sqrt{KN})$, 
\editSF{the refined upper bound $	m(X_N; p; 1/2, 3/2)\leq C_p  \Phi_p(KN)$ holds with}
\begin{equation} \label{8-4} \Phi_p(t):= \begin{cases}
t\log t,   &\  \ \text{if $p=1$, $2$},\\
t(\log t) (\log\log t )^2 ,   &\  \ \text{if $1<p< 2$},\end{cases}\   \ t\ge 2.
\end{equation}

\begin{Theorem}	[\cite{DKT}]\label{thm-8-1}
Given   $1\leq p\leq 2$,   $0<\epsilon\leq \frac12$, $K\ge 2$ and any $N$-dimensional subspace
$X_N\in \textnormal{NI}_{2, \infty}(\sqrt{KN})$,
there exist
$ \xi^1,\ldots, \xi^m\in\Omega$  with $m\leq C_p(\epsilon)\Phi_p(KN)$
such that
\begin{equation}\label{eq-00}
(1-\epsilon) \|f\|^p_{L_p(\Omega, \mu)} \leq  \frac 1m \sum_{j=1}^m  |f(\xi^j)|^p \leq (1+\epsilon) \|f\|_{L_p(\Omega, \mu)}^p,\   \ \forall f\in X_N,
\end{equation}		
where $\Phi_p$ is given in \eqref{8-4},
$C_p(\epsilon)= C \epsilon^{-2} $ for $p=1,2$, and $C_p(\epsilon)=C_p\epsilon^{-2}\log^3 \frac 1\epsilon$ for $1<p<2$.
\end{Theorem}

{
Historical discussions on sampling discretization for the cases $p=2$ and
$1 \le p < 2$ can be found in Subsections~{\bf D.15} and~{\bf D.16} of~\cite{KKLT},
respectively. Here we only mention that the best previously known results are as follows.
It was proved in~\cite{DPSTT2} that, under the $(2,\infty)$--Nikol'skii inequality assumption
$X_N\in\textnormal{NI}_{2,\infty}(\sqrt{K N})$, the discretization estimate~\eqref{eq-00}
holds for $1 \le p < 2$ provided that
$m \ge C(p,K,\varepsilon)\, N (\log N)^3$.
This estimate on $m$ was  improved to $m\ge C(p, K, \varepsilon) N(\log N)^2$ for $1<p<2$ and to $m\ge C(p, K, \varepsilon)N^{p/2}\log N$ for $p>2$ in ~\cite{Ko21}.
Furthermore, sampling discretization for $p>2$ under the $(p,\infty)$-Nikol'skii inequality $X_N\in\textnormal{NI}_{p, \infty}((KN)^{1/p})$ was studied in~\cite{Ko21}, with subsequent improvements for $p>3$ presented in~\cite{DT1}. More recently, \cite{DKM} improved the sample size estimates from \cite{Ko21, DT1} under the $(p,\infty)$-Nikol'skii inequality for  $p>2$, proving that  the discretization inequality~\eqref{eq-00} holds   whenever  $m\ge C(p, K, \varepsilon) N(\log N) (\log\log N)^{p-1}$.

{
Combining Theorem~\ref{thm-8-1} with Lewis' change of density lemma
(see Lemma \ref{lem-6-2}), and following the proof of Corollary~\ref{cor-6-2},
we obtain the following weighted discretization result.
}    		
    		
\begin{Corollary}\label{cor-8-1a}
Given $1\leq p\leq 2$, $0<\epsilon\leq \frac12$\editSF{, and} 
any $N$-dimensional subspace $X_N\subset L_p(\Omega, \mu)$,
there are  a finite set of points  $\{\xi^1,\ldots,\xi^m\}\subset \Omega$  with  $m\leq C_p(\epsilon)\Phi_p(KN)$
and  a set  of nonnegative weights $\{\lambda_j\}_{j=1}^m$
such that  for any $f\in X_N$,
\[
{(1-\epsilon)\|f\|_{L_p(\Omega,\mu)}^p \leq  \sum_{j=1}^m \lambda_j|f(\xi^j)|^p \leq  (1+\epsilon)\|f\|_{L_p(\Omega,\mu)}^p,}
\]		
where $\Phi_p$ is given in \eqref{8-4},
$C_p(\epsilon)= C \epsilon^{-2} $ for $p=1,2$, and $C_p(\epsilon)=C_p\epsilon^{-2}\log^3 \frac 1\epsilon$ for $1<p<2$.
\end{Corollary}

We give a brief description of the scheme of the proof of Theorem \ref{thm-8-1}, which can be  divided into three steps.\\

{\bf Step 1.}\  \ Preliminary discretization.\\

In this step,  we establish results on simultaneous discretization of the $L_2$ and $L_p$ norms, using   $M \le C_p \e^{-r_1}N(\log N)^{r_2}$ points.  This step allows us to reduce the original problem of discretization to a problem in $\RR^M$.
The main ingredient in this step is the following lemma, which follows directly from Proposition \ref{prop-6-2}.

\begin{Lemma}\label{p1L1} Let   $1\leq p<2$ and $ 0< \epsilon_0 < 1/4$. Let $X_N$ be a subspace of $L_\infty(\Omega)$ of dimension $N$ such that $X_N\in \textnormal{NI}_{2, \infty}(\sqrt{KN})$ for some $K\ge 2$.    Then  there exists a finite set of points
$x_1,\ldots, x_m\in  \Omega$   with
$	m\leq C_p\epsilon_0^{-8} K N(\log (KN))^{3}$
such that for  any $f\in X_N$,  we have
\[
(1-\epsilon_0) \|f\|_{{L_p(\Omega, \mu)}}^p \leq \frac 1m \sum_{j=1}^m |f(x_j)|^p  \leq (1+\epsilon_0) \|f\|_{{L_p(\Omega, \mu)}}^p,\label{8-6:2024}
\]
and
\[
(1-\epsilon_0) \|f\|_{{L_2(\Omega, \mu)}}^2 \leq \frac 1m \sum_{j=1}^m |f(x_j)|^2  \leq (1+\epsilon_0) \|f\|_{{L_2(\Omega, \mu)}}^2.\label{8-7:2024}
\]
\end{Lemma}

{\bf Step 2.}\  \  Estimate  the expectation $	\EE \Bl( \sup_{f\in X_N\cap B_p^M} \Bl| \sum_{j=1}^M { \varepsilon_j}  |f(j)|^p\Br| \Br)$, where
$\{ \varepsilon_j\colon  j=1,2,\ldots, M\}$ denotes
a sequence  of independent Bernoulli random variables taking
values $\pm 1$ with probability $1/2$,  and $B_p^M:=\{f\in \RR^M\colon  \|f\|_{\ell_p^M}\!\leq\! 1\}$.	
\\
    	
For $p=1,2$, we have the following lemma, which   was proved by Rudelson \cite[Lemma 1]{Rud1} for $p=2$, and
by   Talagrand \cite{Ta90} for $p=1$
(see also \cite[Theorem 3.1]{DKT}, \cite[Theorem 13]{JS} and  \cite[Proposition 15.16]{LedTal}).

\begin{Lemma}\label{p1T2} Let $X_N$ be a subspace of $\RR^M$ of dimension at most $N$ satisfying
\begin{equation*}\label{8-8:2024}
\|f\|_{\infty}\le \sqrt{KN} \|f\|_{L_2^M},\     \  \   \ \forall f\in X_N
\end{equation*}
for a constant $K\ge 1$.
Then for $p=1$ and $p= 2$ we have
\[
\EE \Bl( \sup_{f\in X_N\cap B_p^M} \Bl| \sum_{j=1}^M { \varepsilon_j}  |f(j)|^p\Br| \Br) \leq C
\sqrt{\frac {KN\log N}M},
\]
where $C$ is a positive absolute constant.
\end{Lemma}

    {The subsequent result for $1 < p < 2$ builds upon the approach established in Talagrand's proof of \cite[Theorem 16.8.2]{Ta}.	}
    	
\begin{Lemma}[{\cite[Theorem 4.2]{DKT}}]\label{lem-8-3}
Let $X_N$ be a  subspace of $\RR^M$ of dimension at most $N$ satisfying
\begin{equation*}\label{5-1-Ta}
\|f\|_{\infty}\le \sqrt{KN} \|f\|_{L_2^M},\     \  \   \ \forall f\in X_N
\end{equation*}
for some constant  $K\ge 1$.
Then for any $p\in (1,2)$, we have
\[
\EE \Bl( \sup_{f\in X_N\cap B_p^M } \Bl| \sum_{j=1}^M 
\editSF{\varepsilon_j |f(j)|^p}
\Br|\Br) \leq C(p) \sqrt{\frac {KN\log M} M} \log \Bl(\frac M{KN}+2\Br).
\]
\end{Lemma}

For $1<p<2$, Talagrand \cite[Theorem 16.8.2]{Ta} proved a result similar to Lemma \ref{lem-8-3}
for a probability measure $\mu$ on $\RR^M$ satisfying $\mu\{j\}\leq \frac 2M$ for $1\leq j\leq M$, under the following  stronger assumption:
for each  orthonormal basis  $\{\varphi_k\}_{k=1}^N$ of $(X_N,\|\cdot\|_{L_2(\mu)})$,     		
\[ \frac 1N \sum_{k=1}^N \varphi_k(j)^2 =1,\   \ j=1,2,\ldots, M.\]

Lemmas~\ref{p1T2} and~\ref{lem-8-3} play crucial roles in the proof of 	Theorem~\ref{thm-8-1}. Roughly speaking, they allow us to reduce the number of points required for good discretization by
{approximately a factor of two.}\\

{\bf Step 3.}\   \ Iteration.\\

To illustrate the idea, we sketch the proof of $p=1$ below.

First, by Lemma \ref{p1L1}  in the preliminary step, without loss of generality, we may assume that \editSF{$\Omega~=~\Omega_M~=~\{1,2,\ldots, M\}$}
and $\mu$ is the uniform probability measure on $\Omega_M$.
For each \editSF{$I~\subset~\Omega_M$}, we denote by $R_I$ the orthogonal projection onto the space spanned by $\{e_i, i\in I\}$, where \editSF{$e_1~=~(1,0,\ldots, 0)$}, $\ldots$, $e_M=(0, \ldots, 0, 1)$ is a canonical basis of $\RR^M$. Thus,  for each $f\in \RR^M$, $(R_I f)(j)=f(j)$ for $j\in I$, and $(R_I f)(j)=0$ for $j\in\Omega_M\setminus I$.
Furthermore, an application of \editSF{Lemma~\ref{p1T2}} yields the following structural result.

\begin{Lemma}[{\cite[Lemma 3.2]{DKT}}]\label{p1L2} Let  $X_N$ be a  subspace of $\RR^M$
of dimension at most $N$ satisfying
\[
\|f\|_{\infty}\le \sqrt{KN} \|f\|_{L_2^M},\     \  \   \ \forall f\in X_N\editSF{,}
\]
for some constant $K\ge 1$.  Let   $J\subset \Omega_M:=\{1, 2,\ldots, M\}$.
Assume that there exist positive  constants $\al_J$, $\bt_J$ such that for any $f\in X_N$ we have for both $p=1$ and $p=2$
\[
\al_J \| f\|^p_{\ell_p^M} \leq \|R_J f\|^p_{\ell_p^M} \leq \bt_J  \| f\|^p_{\ell_p^M}.
\]
Then there exists a subset $I\subset J$ with
\begin{equation*}\label{8-13:2024}
\frac {|J|} 2 \Bl(1-\frac 1 {\sqrt{|J|}} \Br) \leq |I|\leq \frac {|J|} 2
\end{equation*}
such that for any $f\in X_N$ we have for both $p=1$ and $p=2$,
\[
\al_I \| f\|^p_{\ell_p^M} \leq \|R_I f\|^p_{\ell_p^M} \leq \bt_I  \| f\|^p_{\ell_p^M}
\]
where
\[
\al_I:=  \frac {(1-\sigma)\al_J }2,\   \    \bt_I:=  \frac {(1+\sigma)\bt_J }2,\   \
\sigma:= C\sqrt{ \frac { K N\log N} {\al_J M }},
\]
and $C$ is an absolute constant.
\end{Lemma}

We can iterate Lemma \ref{p1L2}  with an appropriate stopping time and finish the proof.

{
\begin{Example}
Let $Q \subset \mathbb{Z}^d$ be an arbitrary set of frequencies, and consider the space
\[
\mathcal{T}(Q) := \operatorname{span}\bigl\{ e^{\mathbf{i} \<k,\cdot\>} : k \in Q \bigr\}
\]
of all trigonometric polynomials with frequencies from $Q$ on the cube
$\Omega = [0, 2\pi)^d$, equipped with the normalized Lebesgue measure.
Then $\mathcal{T}(Q) \in \textnormal{NI}_{2,\infty}(\sqrt{N})$ with
$N := \dim \mathcal{T}(Q) = |Q|$
and Theorem \ref{thm-8-1} implies
that
there exist $m \leq C \epsilon^{-2} N\log N$
points $\xi^1, \ldots, \xi^m\in \Omega$
such that
\[
(1-\epsilon)\|f\|_{L_1(\Omega, \mu)} \leq \frac{1}{m} \sum_{j=1}^m |f(\xi^j)| \leq (1+\epsilon) \|f\|_{L_1(\Omega, \mu)},\   \ \forall f\in \mathcal{T}(Q).
\]
Similarly, for $1<p<2$, there exist
$m \leq C_p(\varepsilon) N\log N(\log\log N)^2$
points $\xi^1, \ldots, \xi^m\in \Omega$
such that
\[
(1-\epsilon)\|f\|_{L_p(\Omega, \mu)}^p \leq \frac{1}{m} \sum_{j=1}^m |f(\xi^j)|^p \leq (1+\epsilon) \|f\|_{L_p(\Omega, \mu)}^p,\   \ \forall f\in \mathcal{T}(Q).
\]
\end{Example}
}    	
    	
\section{Sampling discretization in the uniform norm}
\label{DD}
    	
    	
In this section, we will survey some recent progress on  sampling discretization of the uniform norm. We assume that $\Omega$ is a compact subset of $\RR^d$ equipped  with a Borel probability \editSF{measure~$\mu$}.
Let $X_N$ denote an $N$-dimensional subspace of the space $\mathcal{C}(\Og)$ of continuous functions on $\Og$. For $A\subset \Og$ and $f\in \mathcal{C}(\Og)$, define $\|f\|_A=\sup_{x\in A} |f(x)|$.
We say \editSF{that} $X_N$  admits an $L_\infty$- Bernstein-type discretization theorem  with parameter  $m\in\NN$
\editSF{and} positive constant $C$ if there exist  $m$ points  $\xi^j\in\Omega$, $j=1,2,\ldots, m$\editSF{,}
for which
\begin{equation}\label{9-1a}
\|f\|_{\Omega} \leq C \max _{1 \leq j \leq m}\left|f\left(\xi^j\right)\right|,\   \ \forall f\in X_N.
\end{equation}
For convenience, we denote by $m(X_N;\infty;  C)$ the minimum number $m$ of points required for the discretization \eqref{9-1a} with a given constant $C>0$.

The first results on discretization of the uniform norm were obtained by Bernstein \cite{B1,B2} (see also \cite{Zy}, Ch.10, Theorem (7.28)). In recent years this problem has been extensively studied in various settings {(e.g. see \cite{DP, KKT, VT168, KoTe}).  We will focus on recent progress on sampling discretization of $L_\infty$ norm.  We refer to the  surveys \cite{DPTT, KKLT}  and  the references therein for historical comments on this problem.
}

\subsection{General bounds of $m(X_N;\infty;  C)$  with $C$ being independent of $N$}

We start with the following general {well known} result
{(see, e.g.,~\cite[Theorem~1.3]{KKT})}, which in particular gives the exponential  upper bound $m(X_N;\infty;  2)\leq 9^N$.

\begin{Theorem}\label{thm-9-1}\footnote{ One of the referees kindly pointed  out  that the bound in this theorem can be improved from $(1+\f 8 \va)^N$ to $(1+\f 4 \va)^N$ by working with an $(\va/2)$-net for the set of the extreme points of the unit ball of the dual space.}
Let $X_N$ be an $N$-dimensional subspace of $L_\infty(\Omega)$. Given any $\epsilon\in (0, 1]$, there exists a sequence  $\left\{\xi^j\right\}_{j=1}^m$ of $m \leq \bigl(1+\frac 8 \epsilon\bigr)^N$ points  in $\Omega$  satisfying
\begin{equation}\label{9-3:2024}
\max_{x\in\Omega}|f(x)| \leq (1+\epsilon) \max _{1\leq j\leq m} \left|f\left(\xi^j\right)\right|,\   \ \forall f\in X_N.
\end{equation}
\end{Theorem}

\begin{proof}
Using \eqref{5-2}, we can find a finite sequence $$\{f_j\}_{j=1}^m\subset B:=\{f\in X_N:\  \|f\|_{L_\infty(\Omega)}\leq 1\}\   \     \text{with}\   \ m\leq \Bigl(1+\frac 8 \epsilon\Bigr)^N $$  such that
$B\subset \bigcup_{j=1}^m  (f_j+\frac \epsilon 4 B)$.  For each $1\leq j\leq m$, let  $\xi^j\in\Omega$ be such that $|f_j(\xi^j)|=\|f_j\|_{L_\infty(\Omega)}$.
For any $f\in X_N$ with $\|f\|_{L_\infty(\Omega)} =1$, we can find an integer $1\leq j\leq m$ such that $\|f_j-f\|_{L_\infty(\Omega)}\leq \frac \epsilon4$, which implies
\[ \|f\|_{L_\infty(\Omega)}\leq \|f_j\|_{L_\infty(\Omega)} +\frac \epsilon 4 \leq |f(\xi^j)| +\frac \epsilon 2 \|f\|_{L_\infty(\Omega)}.\]
The estimate \eqref{9-3:2024} then follows since $(1-\frac \epsilon 2)^{-1} \leq 1+\epsilon$.	
\end{proof}

Theorem \ref{thm-9-1} establishes  an  exponential upper bound  in  the dimension $N$ for the number of points  required to discretize the uniform norm in $X_N$.
In general, the discretization \eqref{9-1a} of the uniform norm in $X_N$
with a constant $C$ independent of $N$ may not be possible unless  the number $m$ of required points  grows exponentially with $N$ {(see, e.g.,
\cite{VT168} and \cite[Theorem 1.2]{KKT})}.
More precisely,  the following result {(see \cite{VT168})}, which  complements
Theorem \ref{thm-9-1}, shows that there exists an $N$-dimensional subspace $X_N\subset C[0, 2\pi]$  for which
\[
m(X_N;\infty;  C)\ge (N/e) e^{cN/C^2}\   \ \text{with $c>0$ being an absolute constant}.
\]

\begin{Theorem} [{\cite{VT168}}]\label{thm-9-2}
Let $X_N:=\spn\bigl\{e^{ik_j x}\colon j=1,2,\ldots, N\bigr\}$, where  $\left\{k_j\right\}_{j=1}^N$ is a lacunary sequence of positive integers satisfying  that $k_1=1$  and   $ k_{j+1} \geq b k_j$,  $ j=1,2,\ldots, N-1$ for some constant $b>1$.     	
If there exist a sequence  $\left\{\xi^j\right\}_{j=1}^m \subset [0, 2\pi)$  of $m$ points  and a constant $C>0$ such that
$$
\max_{x\in [0, 2\pi]}|f(x)| \leq C \max _j\left|f\left(\xi^j\right)\right|,  \quad \forall f \in X_N,
$$
then one must have
$$
m \geq(N / e) e^{c N / C^2}
$$
with a positive constant $c$ which may only depend on $b$.
\end{Theorem}

\subsection{Estimates of $m(X_N;\infty;  C_N)$  with $C_N$ being dependent on $N$}

Theorems   \ref{thm-9-1}  and  \ref{thm-9-2}  provide upper and lower estimates of  the number of sample points needed for the discretization inequality \eqref{9-1a} with $C$ being independent of $N$.
We now mention some results, where the discretization constant $C$ in \eqref{9-1a}  is allowed to depend on $N$.
In the classical setting of univariate algebraic and trigonometric polynomials,
this problem was extensively studied (see, e.g.,~\cite{CR92, EZ64}). However, as is customary in this
survey, our goal is to consider a more general framework of abstract finite-dimensional
spaces.

We start with the following result, which is {known as the Auerbach lemma when $\ell=1$} (see \cite[Proposition 1.2.3]{Nov}).

\begin{Theorem}  \label{thm-9-3}
Let $X_N$ be an  $N$-dimensional real subspace of $\mathcal{C}(\Omega)$.  Given
a positive integer $\ell$,  let
\[ X_N(\ell):=\spn\Bl\{ f_1f_2\cdots f_\ell:   \  \ f_1,\ldots, f_\ell\in X_N\Br\},\]
and
$N_\ell: =\dim X_N(\ell)$.
Then there exists  a set $\left\{\xi^j\right\}_{j=1}^{N_\ell}$ of  points in $\Omega$  such that
$$
\max_{x\in\Omega}|f(x)| \leq (N_\ell)^{\frac 1 {\ell}}  \max _{1\leq j\leq N_\ell}  \left|f\left(\xi^j\right)\right|,\    \   \ \forall f\in X_N.
$$
\end{Theorem}
\begin{proof}
Given a  basis $\{ \varphi_1,\ldots, \varphi_{N_\ell}\}\subset X_N(\ell)$  of  $X_N(\ell)$, we define a function $\Delta: \Omega^{N_\ell}\to \RR $ by
$$\Delta (x_1, x_2,\ldots, x_{N_\ell}):=\det \left [ \begin{matrix}
\varphi_1(x_1) & \varphi_2(x_1)&\cdots &\varphi_{N_\ell }(x_1)\\
\varphi_1(x_2) & \varphi_2(x_2)&\cdots &\varphi_{N_\ell }(x_2)\\
\vdots &\vdots & \cdots & \vdots\\
\varphi_{1} (x_{N_\ell} ) & \varphi_{2}(x_{N_\ell})&\cdots &\varphi_{N_\ell} (x_{N_\ell})
\end{matrix}\right]$$
for  $x_1, \ldots, x_{N_\ell}\in\Omega$.
Since $\Delta$ is a continuous function on the compact set $\Omega^{N_\ell}$, there must exist  $\xi^1, \ldots, \xi^{N_\ell}\in\Omega$   such that
$$\Delta(\xi^1,\ldots, \xi^{N_\ell} )\equiv \Delta_{N,\ell}: =\max_{x_1, \ldots, x_{N_\ell} \in\Omega} \Delta(x_1, \ldots, x_{N_\ell}).$$
It is easily seen that $\Delta_{N,\ell}>0$.
For $1\leq j\leq N_\ell$, define  $L_j: \Omega\to \RR$ by
$$L_{j} (x) =\frac{\Delta(\xi^1,\ldots, \xi^{j-1}, x, \xi^{j+1},\ldots, \xi^{N_\ell} )} {\Delta_{N,\ell}},\   \ x\in\Omega.$$
Clearly,  $L_j\in X_N(\ell)$,  $L_j(\xi^k) =\delta_{k,j}$ and $\|L_j\|_{L_\infty(\Omega)}= 1$ for  $1\leq j, k\leq N_\ell$. Moreover, for each $f\in X_N$, we have
\[ \bigl(f(x)\bigr)^\ell  =\sum_{j=1}^{N_\ell} f(\xi^j)^\ell  L_j(x),\  \ x\in\Omega.\]
This implies that for each $f\in X_N$,
\[\|f\|_{L_\infty(\Omega)}^\ell \leq  N_\ell \max_{1\leq j\leq N_\ell}|f(\xi^j)|^\ell.\]
Hence,  \[ \|f\|_{L_\infty(\Omega)}\leq ( N_\ell)^{\frac 1\ell} \max_{1\leq j\leq N_\ell}|f(\xi^j)|,\   \ \forall f\in X_N. \]
{This completes the proof.}
\end{proof}
We give a few remarks on Theorem \ref{thm-9-3}.
\begin{Remark} \textnormal{(i)} Of particular interest is the case when $\ell=1$, where Theorem \ref{thm-9-3} gives the discretization inequality with  the minimum number of points:
\begin{equation} \label{9-3b}	\max_{x\in\Omega}|f(x)| \leq N   \max _{1\leq j\leq N}  \left|f\left(\xi^j\right)\right|,\    \   \ \forall f\in X_N.\end{equation}
Using the notation introduced in the beginning of this section, this implies $m(X_N;\infty;  N)=N$.
{
This  is a  well known result, for which the above proof can be found in the book of E. Novak  \cite[Proposition 1.2.3]{Nov}}.

\textnormal{(ii)}{ For any positive integer $\ell$, we have
\[ N_\ell \leq \binom{N-1+\ell}\ell \leq \Bl( \frac { e (N-1+\ell)}\ell\Br)^\ell.\]
However, in many cases, this last estimate is far from being optimal. For example, if $X_N=\Pi_n^d$ is the space of all algebraic polynomials of total degree at most $n$ in $d$ variables, and $\Omega$ is a compact subset of $\RR^d$ with  nonempty interior, then
$$ N= \dim \Pi_n^d=\binom{n+d}{d}\   \ \text{ and }\  \
N_\ell \leq  \dim\Pi_{n\ell}^d.$$
As a result, we have
\[ \f {N_\ell} N\leq \f{\binom{n\ell+d}{d}}{\binom{n+d}{d}}=\f{ (n\ell+d) (n\ell+d-1)\cdots (n\ell+1)}{(n+d) (n+d-1) \cdots (n+1)}=\ell^d \prod_{j=1}^d \f{n +\f j\ell} {n+j}\leq \ell^d. \]
Assume in addition  that $n\ge d^2$ and $d\ge 2$. Then
\[ N=\prod_{j=1}^d \Bigl(1+\f nj\Bigr) \ge \Bigl(1+\f nd\Bigr)^d\ge d^d.\]
 Now we specify the  integer $\ell\in\NN$ such  that $$\log N  -1\leq \ell<\log N.$$ Then $\ell\ge d\log d -1$,
\[ N_\ell \leq \ell^d N\leq N(\log N)^d,\]
and
\[ N_\ell^{\f 1\ell} \leq  (\ell^d N)^{\f 1\ell} =\exp \Bl( \f d \ell \log \ell +\f1\ell \log N\Br)\leq C_1, \]
where $C_1>0$ is an absolute constant (independent of $d$).
Thus,   Theorem \ref{thm-9-3} provides      a set $\{\xi^j\}_{j=1}^m$ of $$m\leq N_\ell\leq N(\log N)^d $$ points in $\Omega$ such that
\[ \max_{x\in\Omega} |f(x)| \leq C_1 \max_{1\leq j\leq m}|f(\xi^j)|,\   \ \forall f\in X_N=\Pi_n^d.\]
This last estimate    was previously established  in \cite[Proposition~23]{BBCL}, with implicit constants depending on~$d$.}\footnote{ {The authors} are grateful to  the referee for noting that the implicit constants can be made dimension-free, rather than depending on $d$ as they did in the {initial} draft.}
    	 	

\end{Remark}


{If the number of sampling points $m$ is relaxed to be of order $N$ rather than exactly equal \editSF{to~$N$},
one can prove a refinement of the discretization inequality~\eqref{9-3b}, in which the
constant $N$ is improved to  $C\sqrt{N}$.
We state the following result, which  follows directly from the proof of \editSF{Theorem~2} of \cite{KPUU2}.}
\begin{Theorem}[{\cite[Theorem 2]{KPUU2}}]\label{BP1b}
Given any real  $N$-dimensional subspace $X_N$ of $\cC(\Omega)$ and any $\epsilon\in (0, 1]$,
there exists a set     $\{\xi^j\}_{j=1}^m\subset \Omega$ of $m \le (1+\va) (N+1)$ points such that
\[
\max_{x\in\Omega}|f(x)| \le   C \sqrt{N}\va^{-1}\max_{1\le j\le m} |f(\xi^j)|,
\]
where $C>0$ is an absolute constant.
\end{Theorem}
For completeness, we include  the  proof from \cite{KPUU2} below.
The argument relies on Theorem \ref{Thm-3-11-sc} and the following  result by J. Kiefer and J. Wolfowitz \cite{KW}.	
    \begin{Lemma} [\cite{KW}] \label{lem-9-2} Given a real $N$-dimensional subspace $X_N$ of $\mathcal{C}(\Omega)$,  there exists a probability measure $\mu$ on $\Omega$ such that for any $f\in X_N$,
\[ \max_{x\in\Omega} |f(x)|\leq \sqrt{N}\|f\|_{L_2(\mu)}.\]
\end{Lemma}
   \begin{proof}[Proof of Theorem \ref{BP1b}]
   Let $X_{N+1}=\operatorname{span}\, \big( X_N\cup\{1\}\big)$. Clearly,  $\dim X_{N+1}\leq N+1$. Applying Lemma \ref{lem-9-2} to the space $X_{N+1}$ instead of $X_N$, we can find a probability measure $\mu$ on $X_{N+1}$ for which
   \[ \max_{x\in\Omega} |f(x)|\leq \sqrt{N+1}\|f\|_{L_2(\mu)},\  \ \forall f\in X_{N+1}.\]
     By Theorem \ref{Thm-3-11-sc} applied to  $b=1+\va$ and the space $X_{N+1}\subset L_2(\mu)$, there exist a set $\{\xi^j\}_{j=1}^m\subset \Og$ of $m\leq (1+\va) (N+1)$ points  and a sequence $\{\ld_j\}_{j=1}^m$ of non-negative weights such that
   \[ \|f\|_{L_2(\mu)}\leq \Bl( \sum_{j=1}^m \ld_j |f(\xi^j)|^2 \Br)^{1/2} \leq C \va^{-1} \|f\|_{L_2(\mu)},\   \ \forall f\in X_{N+1}.\]
   Since the constant function $1$ belongs to the space $ X_{N+1}$, this implies that
   \[ (\sum_{j=1}^m \ld_j )^{\f12}\leq  C \va^{-1}.\]
   It then follows from Lemma \ref{lem-9-2}  that for any $f\in X_N$,
   \editSF{
   \eq{ \max_{x\in\Og}|f(x)|&\leq (N+1)^{\f12} \|f\|_{L_2(\mu)}\leq (N+1)^{\f12} \Bl( \sum_{j=1}^m \ld_j |f(\xi^j)|^2 \Br)^{\f12}\\
   &\leq (N+1)^{\f12} \max_{1\leq j\leq m} |f(\xi^j)|  \Bl( \sum_{j=1}^m \ld_j  \Br)^{\f12}\leq C \va^{-1}N^{\f12}.
   \qedhere}}
   \end{proof}

    Replacing Theorem \ref{Thm-3-11-sc} with Corollary \ref{cor-3-10} in the argument above yields the following result.

\begin{Theorem}\label{BP1c}
Given any {real}  $N$-dimensional subspace $X_N$ of $\cC(\Omega)$ and any $\epsilon\in (0, 1)$,
there exists a set     $\{\xi^j\}_{j=1}^m\subset \Omega$ of $m \le C \epsilon^{-2}  N$ points such that
\begin{align*}
\max_{x\in\Omega}|f(x)| & 		\leq
(1+\va)  \sqrt{N}\max_{1\le j\le m} |f(\xi^j)|,
\end{align*}
where $C>1$ is an absolute constant.
\end{Theorem}
    	 	
We point out that Theorems \ref{BP1b} and \ref{BP1c} were  proved in \cite[Theorem~6.6]{KKT} and \cite{KoTe}  under the extra condition
$X_N \in \textnormal{NI}_{2,\infty} (H)$ for some constant $H>0$.

    \section*{Acknowledgments}
     The authors sincerely thank the anonymous referees for their time and insightful comments.  Their suggestions have been fully incorporated into the revised manuscript.

\bibliographystyle{plain}
\bibliography{references}

{

\bigskip
\hskip1.4 em\vbox{\noindent 
\editSF{Feng~Dai \\ Department of Mathematical and Statistical Sciences\\ University of Alberta\\
Edmonton, Canada}\\
 {\tt fdai@ualberta.ca}
}
}
{

\bigskip
\hskip1.4 em\vbox{\noindent 
\editSF{Egor~.D.~Kosov \\ 
Centre de Recerca Matem\`atica \\
Campus de Bellaterra\\
Barcelona, Spain}\\
 {\tt kosoved09@gmail.com}
}
}

{

\bigskip
\hskip1.4 em\vbox{\noindent 
\editSF{Vladimir~N.~Temlyakov \\ 
Steklov Mathematical Institute\\ 
Russian Academy of Sciences\\
Moscow, Russia\\
and\\
Department of Mechanics and Mathematics\\ 
Lomonosov Moscow State University\\
Moscow, Russia\\
and\\
Department of Mathematics\\
University of South Carolina\\
Columbia, USA}\\
 {\tt temlyakovv@gmail.com}
}
}

\endddoc
\end{document}